\documentclass[11pt,reqno]{amsart}

\usepackage{amsfonts,amsbsy,amsmath,amssymb, amsthm}
\usepackage{verbatim} 
\usepackage{eufrak}
\usepackage{upgreek}
\usepackage{graphics}
\usepackage{color}
\usepackage{bbm}
\usepackage[usenames,dvipsnames]{xcolor}
\usepackage{enumitem}
\usepackage{mathtools}
\usepackage{fullpage}
\usepackage[T1]{fontenc} 
\usepackage{lipsum}                     
\usepackage{xargs} 
\usepackage{microtype}
\usepackage[normalem]{ulem}
\usepackage{mathrsfs}
\usepackage{stackrel}
\usepackage{scrextend}
\usepackage{bbm}
\usepackage[most]{tcolorbox}

\deffootnote[1em]{1em}{1em}{\textsuperscript{\thefootnotemark}\,}
\allowdisplaybreaks

\definecolor{myMaroon}{HTML}{720E0E}
\definecolor{myBlue}{HTML}{1A5276}
\definecolor{myDarkerBlue}{HTML}{154360}
\usepackage[colorlinks=true, citecolor=myBlue, urlcolor = purple, linkcolor=myMaroon, anchorcolor = myBlue]{hyperref}
\usepackage[colorinlistoftodos,prependcaption,textsize=footnotesize]{todonotes}

\usepackage[colorinlistoftodos,prependcaption,textsize=footnotesize]{todonotes}

\newcommandx{\cmt}[2][1=]{\todo[linecolor=red,backgroundcolor=red!25,bordercolor=red,#1]{#2}}

\numberwithin{equation}{section}
 
\let\eps=\varepsilon
\let\theta=\vartheta
\let\rho=\varrho
\let\sigma=\varsigma
\let\phi=\varphi

\newcommand{\sm}{\setminus}
\newcommand{\frob}{\mathsf{F}}
\newcommand{\R}{\mathbb{R}}

\newcommand{\bx}{\boldsymbol{x}}
\newcommand{\ba}{\boldsymbol{a}}
\newcommand{\bu}{\boldsymbol{u}}
\newcommand{\bv}{\boldsymbol{v}}
\newcommand{\bz}{\boldsymbol{z}}

\newcommand{\by}{\boldsymbol{y}}
\newcommand{\br}{\boldsymbol{r}}
\newcommand{\bm}{\boldsymbol{m}}
\newcommand{\bh}{\boldsymbol{h}}
\newcommand{\bmu}{\boldsymbol{\mu}}
\newcommand{\bpsi}{\boldsymbol{\psi}}
\newcommand{\bxi}{\boldsymbol{\xi}}
\newcommand{\bphi}{\boldsymbol{\phi}}
\newcommand{\supp}{\mathrm{supp}}
\DeclareMathOperator*{\argmin}{argmin}
\renewcommand{\Pr}{\mathbb{P}}
\newcommand{\Ex}{\mathbb{E}}
\newcommand{\inner}[2]{\left\langle #1, #2 \right\rangle}
\newcommand{\ind}{\mathbbm{1}}
\makeatletter
\def\moverlay{\mathpalette\mov@rlay}
\def\mov@rlay#1#2{\leavevmode\vtop{%
   \baselineskip\z@skip \lineskiplimit-\maxdimen
   \ialign{\hfil$\m@th#1##$\hfil\cr#2\crcr}}}
\newcommand{\charfusion}[3][\mathord]{
    #1{\ifx#1\mathop\vphantom{#2}\fi
        \mathpalette\mov@rlay{#2\cr#3}
      }
    \ifx#1\mathop\expandafter\displaylimits\fi}
\makeatother

\newcommand{\scaps}[1]{{\scshape #1}}
\newcommand{\bscaps}[1]{\textsc{\textbf{#1}}}

\def\QED{$\blacksquare$}
\def\inQED{$\square$}

\renewenvironment{proof}
{\vspace{1ex}\noindent{\sl Proof.}\hspace{0.5em}}{\hfill \QED \vspace{1ex}}

\newenvironment{proofof}[1]
{\vspace{1ex}\noindent\scaps{Proof of #1.}\hspace{0.5em}}{\hfill \QED \vspace{1ex}}

\newcounter{thmcnt}
\counterwithin{thmcnt}{section}

\newenvironment{theorem}
{
\refstepcounter{thmcnt} 
\vspace{-1ex}
\ \\
\noindent
\begin{it}
\noindent
\bscaps{Theorem~\thethmcnt.}\hspace{-0.5ex}
}
{\end{it} \vspace{1ex}}

\newenvironment{lemma}
{
\refstepcounter{thmcnt} 
\vspace{-1ex}
\ \\
\noindent
\begin{it}
\noindent
\bscaps{Lemma~\thethmcnt.}\hspace{-0.5ex}
}
{\end{it} \vspace{1ex}}

\newenvironment{observation}
{
\refstepcounter{thmcnt} 
\vspace{-1ex}
\ \\
\noindent
\begin{it}
\noindent
\bscaps{Observation~\thethmcnt.}\hspace{-0.5ex}
}
{\end{it} \vspace{1ex}}

\newenvironment{corollary}
{
\refstepcounter{thmcnt} 
\vspace{-1ex}
\ \\
\noindent
\begin{it}
\noindent
\bscaps{Corollary~\thethmcnt.}\hspace{-0.5ex}
}
{\end{it} \vspace{1ex}}

\newenvironment{proposition}
{
\refstepcounter{thmcnt} 
\vspace{-1ex}
\ \\
\noindent
\begin{it}
\noindent
\bscaps{Proposition~\thethmcnt.}\hspace{-0.5ex}
}
{\end{it} \vspace{1ex}}

\newenvironment{remark}
{
\refstepcounter{thmcnt} 
\vspace{-1ex}
\ \\
\noindent
\bscaps{Remark~\thethmcnt.}\hspace{-0.5ex}
}
{\vspace{1ex}}

\makeatletter
\def\section{\@ifstar\unnumberedsection\numberedsection}
\def\numberedsection{\@ifnextchar[
  \numberedsectionwithtwoarguments\numberedsectionwithoneargument}
\def\unnumberedsection{\@ifnextchar[
  \unnumberedsectionwithtwoarguments\unnumberedsectionwithoneargument}
\def\numberedsectionwithoneargument#1{\numberedsectionwithtwoarguments[#1]{#1}}
\def\unnumberedsectionwithoneargument#1{\unnumberedsectionwithtwoarguments[#1]{#1}}
\def\numberedsectionwithtwoarguments[#1]#2{%
  \ifhmode\par\fi
  \removelastskip
  \vskip 1.7ex\goodbreak
  \refstepcounter{section}%
  \begingroup
  \noindent\leavevmode\Large\bfseries\scshape\normalsize
  \begin{center} \thesection.\ #2\end{center} 
  \endgroup
  \addcontentsline{toc}{section}{%
    \protect\numberline{\bf \thesection.}%
    \hspace{2.5ex} #1}%
  }
\def\unnumberedsectionwithtwoarguments[#1]#2{%
  \ifhmode\par\fi
  \removelastskip
  \vskip 1.7ex\goodbreak
  \begingroup
  \noindent\leavevmode\Large\bfseries\scshape\centering 
  \begin{center} #2 \end{center} \par
  \endgroup
  \vskip 2ex\nobreak
  \addcontentsline{toc}{section}{%
    \hspace{1ex} #1}%
  }
\makeatother

\makeatletter
\def\subsection{\@ifstar\unnumberedsubsection\numberedsubsection}
\def\numberedsubsection{\@ifnextchar[
  \numberedsubsectionwithtwoarguments\numberedsubsectionwithoneargument}
\def\unnumberedsubsection{\@ifnextchar[
  \unnumberedsubsectionwithtwoarguments\unnumberedsubsectionwithoneargument}
\def\numberedsubsectionwithoneargument#1{\numberedsubsectionwithtwoarguments[#1]{#1}}
\def\unnumberedsubsectionwithoneargument#1{\unnumberedsubsectionwithtwoarguments[#1]{#1}}
\def\numberedsubsectionwithtwoarguments[#1]#2{%
  \ifhmode\par\fi
  \removelastskip
  \vskip 1.7ex\goodbreak
  \refstepcounter{subsection}%
  \noindent
  \leavevmode
  \begingroup
  \bfseries\normalsize
  \noindent \thesubsection\ \bscaps{#2.}\ 
  \endgroup
  \addcontentsline{toc}{subsection}{%
    \hspace{2ex}\protect\numberline{\bf \thesubsection.}%
    \hspace{1ex} #1}%
  }
\def\unnumberedsubsectionwithtwoarguments[#1]#2{%
  \ifhmode\par\fi
  \removelastskip
  \vskip 3ex\goodbreak
  \noindent
  \leavevmode
  \begingroup
  \bfseries\normalsize
  \begin{center}\bscaps{#2.} \end{center}
  \endgroup
  \addcontentsline{toc}{subsection}{%
    \hspace{1ex} #1}%
  }
\makeatother

\makeatletter
\def\subsubsection{\@ifstar\unnumberedsubsubsection\numberedsubsubsection}
\def\numberedsubsubsection{\@ifnextchar[
  \numberedsubsubsectionwithtwoarguments\numberedsubsubsectionwithoneargument}
\def\unnumberedsubsubsection{\@ifnextchar[
  \unnumberedsubsubsectionwithtwoarguments\unnumberedsubsubsectionwithoneargument}
\def\numberedsubsubsectionwithoneargument#1{\numberedsubsubsectionwithtwoarguments[#1]{#1}}
\def\unnumberedsubsubsectionwithoneargument#1{\unnumberedsubsubsectionwithtwoarguments[#1]{#1}}
\def\numberedsubsubsectionwithtwoarguments[#1]#2{%
  \ifhmode\par\fi
  \removelastskip
  \vskip 3ex\goodbreak
  \refstepcounter{subsubsection}%
  \noindent
  \leavevmode
  \begingroup
  \bfseries
  \thesubsubsection\ \bscaps{#2.}\  
  \endgroup
  \addcontentsline{toc}{subsubsection}{%
    \hspace{2ex} \protect\numberline{\bf \thesubsubsection.}%
    #1}%
  }
\def\unnumberedsubsubsectionwithtwoarguments[#1]#2{%
  \ifhmode\par\fi
  \removelastskip
  \vskip 3ex\goodbreak
  \noindent
  \leavevmode
  \begingroup
  \bfseries
  \bscaps{#2.}\  
  \endgroup
  \addcontentsline{toc}{subsubsection}{%
     #1}%
  }
\makeatother

\begin{document}

\title{\bscaps{Smoothed analysis in compressed sensing}}
\author{Elad Aigner-Horev}
\author{Dan Hefetz}
\author{Michael Trushkin}
\maketitle

\begin{abstract}
Arbitrary matrices $M \in \R^{m \times n}$, randomly perturbed in an additive manner using a random matrix $R \in \R^{m \times n}$, are shown to asymptotically almost surely satisfy the so-called {\sl robust null space property}. Whilst insisting on an asymptotically optimal order of magnitude for $m$ required to attain {\sl unique reconstruction} via $\ell_1$-minimisation algorithms, our results track the level of arbitrariness allowed for the fixed seed matrix $M$ as well as the degree of distributional irregularity allowed for the entries of the perturbing matrix $R$. Starting with sub-gaussian entries for $R$, our results culminate with these allowed to have substantially heavier tails than sub-exponential ones. Throughout this trajectory, two measures control the arbitrariness allowed for $M$; the first is $\|M\|_\infty$ and the second is a localised notion of the Frobenius norm of $M$ (which depends on the sparsity of the signal being reconstructed). A key tool driving our proofs is {\sl Mendelson's small-ball method} ({\em Learning without concentration}, J. ACM, Vol. $62$, $2015$).  
\end{abstract} 

\bigskip
\begin{center}
\end{center}

\section{{\lsstyle Introduction}}

\medskip
\noindent
\scaps{{\lsstyle Synopsis}.} A canonical problem in the field of {\sl Compressed Sensing} is that of designing efficient algorithms for the reconstruction of an undisclosed vector $\bx \in \R^n$ through a measurements vector $M\bx$, or a noisy version thereof, where $M \in \R^{m \times n}$. Here $m$ is the number of measurements taken; its minimisation is of great significance. A robust theory now exists (see, e.g.,~\cite{FR13} as well as our account below for details) pinpointing various matrical properties whose satisfaction by $M$ allows for the exact and/or approximate reconstruction of $\bx$. In that, some two strands of matrical properties have become dominant. The first is the so-called {\sl Restricted Isometry Properties} (RIPs, hereafter) and the second is that of {\sl Null Space Properties} (NSPs, henceforth). These matrical properties are found conducive for the aforementioned task of reconstructing undisclosed vectors.  
The verification of whether a given matrix $M$ satisfies said properties is computationally intractable. In contrast, these properties are, nevertheless, ubiquitous as a plethora of probabilistic constructions demonstrates; in that, various models of random matrices $M$ are known to asymptotically almost surely (a.a.s., hereafter) satisfy said properties and consequently are compatible with various types of reconstructions of undisclosed vectors.

We study the NSPs of matrices of the form $M+R$, where $M \in \R^{m \times n}$ is an arbitrary deterministic matrix and $R \in \R^{m \times n}$ is a random matrix. In that, we pursue both the traits that the (deterministic) so-called {\em seed} matrix $M$ should satisfy as well as trace the extent of generality (or structural irregularity) that the distributions of the entries of the perturbing matrix $R$ can be allowed to have whilst maintaining that $M+R$ a.a.s. satisfies the {\sl robust null space property} (see the definition below); all this whilst keeping $m$ - the number of measurements - asymptotically best possible.

Roughly put, yet made precise below in Section~\ref{sec:results}, the following are some core messages arising from our results.   

\begin{enumerate}
    \item [{\bf 1.}] If the entries of $R$ are independent and sub-gaussian and further still sub-exponential, then $M+R$ a.a.s. satisfies the robust null space property with an optimal $m$ provided $\|M\|_\infty$ is independent of $n$ and $s$ (the sparsity level of the signal to be reconstructed), as long as $s \leq n^{1-\eps}$ holds for some arbitrarily small yet fixed $\eps > 0$. For details, see Corollary~\ref{cor:classical-perturb}.  
    
    \item [{\bf 2.}] In fact, the above requirement imposed on $M$ can be replaced by allowing $\|M\|_{\frob,s} = O\left(\sqrt{s \log(en/s)}\right)$ to hold, where $\|M\|_{\frob,s}$ is a certain {\sl localised} notion of the Frobenius norm of $M$ defined in~\eqref{eq:local-frob}. 
    
    \item [{\bf 3.}] All of the above remains true in the significantly more general setting in which the entries of $R$ are allowed to have tails substantially heavier than sub-exponential; for details, see Theorem~\ref{thm::main-perturb}.     
\end{enumerate}


Our results then enhance the aforementioned ubiquity claim made for NSPs by demonstrating that rather arbitrary matrices can be ``mended" through additive random perturbations as to a.a.s. yield a matrix satisfying a null space property of interest. In that, we apply the so-called framework of  {\sl Smoothed Analysis} to cornerstone notions in Compressed Sensing.

\bigskip
\noindent
\scaps{{\lsstyle Smoothed Analysis}.} The study of the properties of deterministic structures, with potentially undesirable features, randomly perturbed with a corresponding stochastic {\sl noise}, is referred to as {\em Smoothed Analysis}. This field originated with the seminal work of Spielman and Teng~\cite{ST04} in the realm of Computer Science; a rather accurate account as to the impact of Smoothed Analysis on Computer Science is provided in the book~\cite{R21}. Since the aforementioned result of Spielman and Teng, Smoothed Analysis has flourished across Combinatorics~\cite{AHDHS22,AHDHS22b,AHH21,AHHK23,AHHK25,AHHL23,AHHS24,AHP19,ADRRS21,ADR23,BHKM19,BFKM04,BFM03,BMPP20,CHT23,DKCM24,DT20,DT20b,DRRS20,KLS24,KST} as well as in Random Matrix Theory~\cite{AHHT25,J21,JSS22,TV10}. 

In this discipline, the resulting typical perturbed structure is viewed as {\sl smoother} compared to the original deterministic structure, forming the {\em seed} for the perturbation, on account of the former typically possessing certain desirable properties absent in the seed. Two facets are traditionally  observed in Smoothed Analysis results. The first being the level of generality enjoyed by the seed structure; the second focuses on the nature of the perturbing noise. Several objectives are of interest in this venue. One may seek to minimise the amount of randomness used for the perturbation. Additionally, there is interest in letting the noise be as unrestricted as possible, thus exploring the most general type of noise distributions for which smoothed analysis is possible. 

\bigskip
\noindent
\scaps{{\lsstyle Compressed Sensing}.}
Given integers $m \leq n$ and a {\sl measurements} matrix $M \in \R^{m \times n}$, a concealed vector $\bx \in \R^n$, as well as a {\sl measurements} vector $\by = M\bx + \bmu \in \R^m$ with $\|\bmu\|_p \leq \eps$ for some $\eps >0$ and $p \geq 1$, the so-called {\em Reconstruction Problem} is a central problem in the field of {\sl Compressed Sensing} calling for the reconstruction of the unseen vector $\bx$ from its revealed measurements $\by$. An influential methodology used for solving the Reconstruction Problem is the so-called {\sl basis pursuit}\footnote{Basis pursuit is a form of convex relaxation of an $\ell_0$-minimisation problem which can be recast as a linear program rendering the relaxation computationally tractable compared to its intractable origin~\cite{N95}.} technique\footnote{Also called basis pursuit de-noising program.} first appearing in~\cite{CDS01}. The latter delivers an approximation $\widehat{\bx} \in \R^n$ of $\bx$ satisfying 
\begin{equation}\label{eq:basis-p}
\widehat{\bx} \in \argmin \|\bz\|_1 \; \text{subject to}\; \|\by - M \bz \|_p \leq \eps.
\end{equation}
If $\eps$, the parameter controlling the {\sl noise} of the samples $\by$, satisfies $\eps = 0$, then the convex optimisation program~\eqref{eq:basis-p} is said to be executed in the {\em de-noised} setting. 

Such minimisation problems are effectively solved (see, e.g.,~\cite{FR13}). The proximity of the solution $\widehat{\bx}$ to $\bx$ determines the quality of the reconstruction; of special interest is the case where the latter two coincide allowing for {\sl exact reconstruction} to take place. Properties of measurements matrices $M$ allowing for high quality and perhaps exact reconstructions are then of main interest in this venue. In particular, the minimisation of $m$, the number of measurements of $\bx$ to be taken through $M$, is of great significance. For the classical choices of $p =1,2$ in~\eqref{eq:basis-p}, exact reconstruction in the de-nonised setting requires 
\begin{equation}\label{eq:optimal-m}
m = \Omega(s \log (n/s))
\end{equation}
by {\sl any} reconstruction algorithm (see~\cite{BIPW10,CDD09} and references therein), whenever $\bx$ is $s$-{\em sparse} by which we mean that $\|\bx\|_0 := |\supp(\bx)| := |\{i \in [n]: \bx_i \neq 0\}| \leq s$. Exact reconstruction algorithms realising~\eqref{eq:optimal-m} carry the message that reconstructing an (adequately) sparse vector essentially requires a number of measurements thereof that exceeds the size of the support of the reconstructed vector by a mere logarithmic factor. 

Two prominent lines of research handling~\eqref{eq:basis-p} have emerged in the field of Compressed Sensing. One originates with the works of Cand\'es and Tao~\cite{CT05,CT06} introducing the notion of {\em restricted isometry properties}. The second entails the notion of {\em null space properties}; its origins are more dispersed, so to speak; consult~\cite[Notes of Chapter~4]{FR13} for an accurate account. Roughly put, RIPs and NSPs are matrix properties, defined below, that if satisfied by the aforementioned measurements matrix $M$, guarantee ``high quality" solutions for~\eqref{eq:basis-p}. Some of these properties enable exact reconstruction and in fact characterise matrices that are able to supply this level of reconstruction; see, e.g.,~\cite[Chapter~4]{FR13} for further details. 

\bigskip
\noindent
\scaps{{\lsstyle RIP}.} A matrix $M \in \R^{m\times n}$ is said to satisfy RIP of order $s \in [n]$ with parameter $\delta$ provided that
\begin{equation}\label{eq:RIP}
(1 -\delta) \|\bx\|_2 \leq \|M\bx\|_2 \leq (1 +\delta) \|\bx\|_2
\end{equation}
holds, whenever $\bx \in \Sigma_s : = \{\bu \in \R^n: \|\bu\|_0 \leq s\}$; in that, $M$ serves as an {\sl approximate isometry} over $\Sigma_s$. Given $s \in [n]$, define $\delta_s(M)$ to be the least $0 < \delta \in \R$ for which $M$ satisfies RIP with parameter $\delta$ for the members of $\Sigma_s$. In fact, RIP is a part of a larger family of properties in which a matrix $M$ must satisfy 
$$
c\|\bx\|_q \leq \|M\bx\|_p \leq C \|\bx\|_q
$$
for some constants $c,C > 0$, whenever $\bx$ is $s$-sparse. Following~\cite{SGR15}, we refer to this form of conditions using the term $\mathrm{RIP}_{p,q}$.

If a matrix $M$ satisfies RIP with $\delta_s(M) \leq 1/3$, then the technique of basis pursuit~\eqref{eq:basis-p}, utilising $M$ as its measurements matrix, can perform exact reconstruction of any $s$-sparse vector in the de-noised setting; see, e.g.~\cite[Theorem~6.9]{FR13}. 

In~\cite{CT05,CT06} it is shown that (scaled) Gaussian matrices\footnote{We refer to random matrices with independent, though not necessarily identical, Gaussian entries as {\em Gaussian matrices}.} with i.i.d. entries a.a.s. satisfy RIP whilst realising~\eqref{eq:optimal-m}; sub-gaussian\footnote{We follow the definition of~\cite[Chapter~2]{Vershynin}.} matrices\footnote{We refer to random matrices with independent, though not necessarily identical, sub-gaussian entries as {\em sub-gaussian matrices}.} with i.i.d entries were handled in~\cite{BDDW08,MPTJ08} (see~\cite[Chapter~9]{FR13} as well). We would be remiss if we were not to mention that rudiments of the arguments needed for establishing RIP for Gaussian and Bernoulli matrices appeared as early as~\cite{GG84,K97}. 
In fact, to the best of our knowledge, the only known constructions of matrices capable of driving basis pursuit~\eqref{eq:basis-p} into exact reconstruction via RIP in the de-noised setting whilst realising~\eqref{eq:optimal-m}, are probabilistic; canonical such matrices are (scaled) sub-gaussian matrices; for additional models consult~\cite[Chapter~12]{FR13}. This is not surprising in view of the computational intractability accompanying the verification of RIP~\cite{BDMS13, TP14} even approximately~\cite{W18}.

\bigskip
\noindent
\scaps{{\lsstyle NSPs}.} A matrix $M \in \R^{m \times n}$ is said to satisfy the {\em null space property relative to a set $S \subseteq [n]$} provided 
$
\|\bv_S\|_1 < \|\bv_{\bar S}\|_1
$
holds for every $\bv \in \mathrm{ker} (M) \sm\{\boldsymbol{0}\}$,  where $\bar S := [n] \sm S$ and where by $\bv_S$ we mean a vector in $\R^n$ whose support coincides with $S$ and satisfies $(\bv_S)_i = \bv_i$, whenever $i \in S$. If $M$ satisfies the null space property relative to every subset $S \subseteq [n]$ of size $|S| \leq s \in [n]$, then $M$ is said to satisfy the NSP of order $s \in [n]$, denoted $\mathrm{NSP}(s)$. By~\cite[Theorem~4.4]{FR13} a matrix $M \in \R^{m \times n}$ allows for exact reconstruction in the de-noised setting of~\eqref{eq:basis-p} if and only if is satisfies NSP (of the appropriate order).  

A discussion pertaining to the connections between RIPs and NSPs appears below; we choose to first proceed with the definitions of two strengthened variants of the NSP that were introduced in~\cite{F14}. The first is referred to as {\em stable} NSP (SNSP, hereafter) and the second is called {\em robust} NSP (RNSP, henceforth) with the latter being the strongest and implying SNSP.   

A matrix $M \in \R^{m \times n}$ is said to satisfy SNSP  with parameter $\rho \in (0,1)$ relative to a subset $S \subseteq [n]$ provided
$
\|\bv_S\|_1 \leq \rho \|\bv_{\bar S}\|_1
$
holds for every $\bv \in \mathrm{ker} (M)$. The matrix $M$ is said to satisfy the SNSP of order $s \in [n]$ and parameter $\rho$, denoted $\mathrm{SNSP}(s,\rho)$, provided the latter has said property relative to every subset $S \subseteq [n]$ of size $|S| \leq s$. Roughly put, the notion of stability was introduced in order to cope with the reconstruction of vectors that are close to sparse ones; see~\cite[Section~4.2]{FR13} for further details.

In order to accommodate the reconstruction of an undisclosed vector $\bx \in \R^n$ through a measurements vector $\by \in \R^m$ satisfying, say, $\|\by - M\bx\|_2 \leq \eta$, for some $\eta := \eta(m,n) \geq 0$, an additional strengthening of the aforementioned null space properties is introduced through the notion of {\sl robustness}. 
A matrix $M \in \R^{m \times n}$ is said to satisfy the $\ell_q$-RNSP with respect to the norm $\|\cdot\|$ along with parameters $\rho \in (0,1)$ and $\tau >0$ and relative to a set $S \subseteq [n]$ provided that 
\begin{equation}\label{eq:rnsp}
\|\bv_S\|_q \leq \frac{\rho}{s^{1-1/q}} \|\bv_{\bar S}\|_1 + \tau \|M\bv\|
\end{equation}
holds for every $\bv \in \R^n$. If said property is upheld by $M$ 
relative to every subset $S \subseteq [n]$ of size $|S| \leq s$, then $M$ is said to satisfy the $\ell_q$-RNSP of order $s$ (with parameters $\rho$ and $\tau$). 
If $\|\cdot\|$ is the $\ell_p$-norm, then we write $(\ell_q,\ell_p)-\mathrm{RNSP}(s,\rho,\tau)$ to indicate the corresponding property; if $p = q$, then $\ell_q-\mathrm{RNSP}(s,\rho,\tau)$ is written instead. The parameters $\rho$ and $\tau$ are suppressed writing $\ell_q-\mathrm{RNSP}(s)$, whenever the focus on $s$ takes precedence. For further details pertaining to RNSP, consult~\cite[Section~4.3]{FR13}; in particular, this property has several variants omitted here. 
As in the case of RIP, computational intractability accompanies the NSPs~\cite{TP14} as well. 

\bigskip
\noindent
\scaps{{\lsstyle NSPs of random matrices}.} For the sake of brevity, in this venue we choose to focus on $\ell_2$-RNSP results solely; this on account of $\ell_q$-RNSP results following in the wake of the former along similar lines. Standard Gaussian matrices are random matrices with each entry forming an independent copy of $N(0,1)$. By~\cite[Theorem~9.29]{FR13} (see~\cite[Theorem~3.3]{SGR15} as well), such a matrix $A \in \R^{m \times n}$ satisfies $\mathrm{SNSP}(s,\rho)$ with probability at least $1-\eps$ whenever $s < n$, $\eps \in (0,1)$, and 
$$
\frac{m^2}{m+1} \geq 2 s \log(en/s) \left(1+\rho^{-1}+0.92 + \sqrt{\frac{\log(\eps^{-1})}{s \log \left(\frac{en}{s} \right)}} \right)^2.
$$

Models more conducive for the computational setting are captured through discrete random matrices; a prominent example of such a model is the family of the $p$-{\em Bernoulli} matrices, that is, matrices whose entries are i.i.d. copies of $\mathrm{Ber}(p)$ for some $p \in [0,1]$. For such a matrix $A \in \{0,1\}^{m \times n}$, a result by K\"ung and Jung, namely \cite[Theorem~9]{KJ16b}, asserts that $A$ satisfies $\ell_2-\mathrm{RNSP}(s,\rho,\tau)$ with failure probability at most $\exp\left(-\frac{(p(1-p))^2}{72}m \right)$, whenever $\rho \in (0,1)$ and
$$
m \geq C \rho^{-2} \alpha(p) s\left(\log \left(\frac{en}{s}\right) +\beta(p) \right),
$$
where $C >0$ is a sufficiently large absolute constant and where for every $p \in (0,1) \setminus \{1/2\}$,
$$
\alpha(p) := \frac{2p-1}{(p(1-p))^3\log\left(\frac{p}{1-p} \right)}, \; \beta(p) := \frac{2p^2\log\left(\frac{p}{1-p} \right)}{2p-1},\; \text{and}\; \tau = \frac{C'}{(p(1-p))^{3/2} \sqrt{m}}
$$
with $C' >0$ being some absolute constant (for $p = 1/2$, a similar result holds). 



The RNSP of random matrices whose entries are {\sl heavy-tailed}, so to speak, has been studied in~\cite{SGR15,LM17}. The latter consider random matrices whose entries are independent copies of a random variable $X$ satisfying 
\begin{equation}\label{eq:beyond-exp}
\|X\|_p:= \left(\Ex \{|X|^p \} \right)^{1/p} \leq \lambda p^\alpha
\end{equation}
for some $\lambda >0$ and $\alpha \geq 1/2$, whenever $2 \leq p \leq \log n$. Sub-gaussian and sub-exponential random variables satisfy~\eqref{eq:beyond-exp} for all of their moments; consequently, the condition~\eqref{eq:beyond-exp} ventures well-beyond these classical distributions allowing for the entries of the random matrix to possess rather heavy tails.   

Concretely, the aforementioned result from~\cite{SGR15} (see Corollary~5.3 there) asserts that random matrices $A \in \R^{m \times n}$ whose entries are zero-mean and satisfying~\eqref{eq:beyond-exp}, and enjoy the additional small-ball property
\begin{equation}\label{eq:small-ball-beyond}
\forall i \in [m] \;\;\; \Pr \left\{|\inner{\ba_i} {\bx}| \geq \zeta \right\} \geq \beta,
\end{equation}
where $\ba_1, \ldots, \ba_m$ are the rows of $A$, $\bx \in \mathbb{S}^{n-1}$, $0 < \zeta \in \R$, and $\beta \in [0,1]$, satisfy $\ell_2-\mathrm{RNSP}(s)$\footnote{To extract details on $\rho$ and $\tau$ in this result, consult~\cite[Theorem~5.1]{SGR15} as well as~\cite[Theorem~A]{LM17}.} with probability at least $1 -\eta$ provided 
\begin{equation}\label{eq:m-beyond-exp}
m \geq C \max\left\{\frac{\lambda^2 \exp(2\alpha-2)}{(\zeta \beta)^2}s\log(en/s), \frac{\log\left(\eta^{-1} \right)}{\beta^2}, (\log n)^{2\alpha -1}\right\},
\end{equation}
where $\lambda$ and $\alpha$ are as in~\eqref{eq:beyond-exp}.

The results of~\cite{SGR15,LM17} essentially assert that random matrices whose entries are mean-zero i.i.d.\footnote{In~\cite{LM17}, this requirement is mitigated further where merely the rows of $A$ are required to have identical distributions. The stricter requirement of i.i.d. entries appearing in~\cite{SGR15} improves certain parameters.} copies of a rather heavy-tailed random variable, which are also endowed with the property that each of their rows exhibits a {\sl small-ball} probability bound over the $n$-dimensional ball, as seen in~\eqref{eq:small-ball-beyond}, are a.a.s. $\ell_2$-RNSP. 


\bigskip
\noindent
\scaps{{\lsstyle Gap between RIPs and NSPs}.} A matrix $M$ satisfying RIP  with parameter $\delta$ for all members of $\Sigma_{2s}$ satisfies $\ell_2-\mathrm{RNSP}(s,\rho(\delta),\tau(\delta))$~\cite[Theorem~6.13]{FR13}\footnote{RIP does not imply $\ell_q$-RNSP when $q >2$; roughly put, the kernel of a matrix satisfying RIP may contain non-sparse ``spiky" vectors that break the $\ell_q$-RNSP property - a matter that cannot occur in the $\ell_2$ geometry.}. Nevertheless, a gap between the NSPs and the RIPs exists; an accurate discussion of said gap can be found in~\cite{CCW16,SGR15} delivering the message that NSPs are weaker than RIP and consequently form less demanding properties. At a more fundamental level, and as explained in~\cite[Page~147]{FR13}, RIP is highly sensitive to various scalings of the measurements matrix $M$. That is, given a matrix $M$ satisfying~\eqref{eq:RIP} with $\delta_s(M) < 3/5$, it can be shown that $2 M$ satisfies~\eqref{eq:RIP} for members of $\Sigma_s$ provided $\delta_{s}(2M) \geq 3 - 4 \delta_s(M) > \delta_s(M)$. This in particular means that $\delta_s(DM) > \delta_s(M)$ may occur for diagonal matrices $D \in \R^{m \times m}$. Reshuffling the measurements of $M$ has no effect on its RIP parameters; indeed, $\delta_s(PM) = \delta_s(M)$ holds, whenever $P \in \R^{m \times m}$ is a permutation matrix. NSPs, on the other hand, exhibit a more stable behaviour in the face of rescaling~\cite[Remark~4.6]{FR13}. Indeed, the {\em kernel} identity $\mathrm{ker} (NM) = \ker (M)$, holding whenever $N \in \R^{m \times m}$ is non-singular, implies that the NSP is preserved under such matrical products on the left provided $M$ satisfies the NSP; the same cannot be said for matrix multiplications to the right of $M$~\cite[Exercise~4.2]{FR13}. 

One type of gap between RIP and the NSPs that we choose to accentuate arises through the study of RIP for sub-exponential matrices\footnote{We refer to random matrices with independent, though not necessarily identical, sub-exponential entries as {\em sub-exponential matrices}.}. The latter satisfy RIP a.a.s. provided that the number of samples (i.e., $m$) taken has order of magnitude $\Omega\left(s \log^2(n/s)\right)$, with $s$ as per~\eqref{eq:optimal-m}; this bound cannot be improved~\cite{ALPTJ11}. A departure from the optimal bound~\eqref{eq:optimal-m} is then observed for RIP in this setting. NSPs, however, are met a.a.s. by sub-exponential matrices~\cite{SGR15,LM17} whilst realising~\eqref{eq:optimal-m} as seen in~\eqref{eq:m-beyond-exp}. 


\bigskip
\noindent
\scaps{{\lsstyle Smoothed analysis in Compressed Sensing}.} To the best of our knowledge, the first to study the NSPs of deterministic matrices $M$, randomly perturbed in an additive fashion using a sub-gaussian noise matrix $R$, were Shadmi, Jung, and Caire~\cite{SJC19b} (see~\cite{SJC19} as well); their work generalises certain aspects seen in~\cite{KJ16b,KJ16} (with all of these results building upon ideas seen in~\cite{SGR15,LM17}). Roughly put, in~\cite{SJC19b} it is shown that a deterministic matrix $M$ with {\sl all} of its entries set to the {\sl same} real constant can be made to satisfy $\ell_q$-RNSP a.a.s. following an additive random perturbation $M+R$, with $R$ being a random matrix having independent sub-isotropic identically distributed rows, and satisfying rather strict sub-gaussian assumptions (see~\cite[Corollary~1]{SJC19,SJC19b} for details).   

Smoothed Analysis of RIP was pursued by Kasiviswanathan and Rudelson~\cite{KR19} along a {\sl multiplicative} random perturbation model. It is shown in~\cite{KR19} (see Theorem~3.1 there) that given $\eps \in (0,1)$, a deterministic matrix $M \in \R^{m \times n}$, and a sub-gaussian matrix $R \in \R^{n \times d}$ with independent centred entries $R_{ij}$ satisfying $\Ex\{R^2_{ij}\}=1$ and $\max_{i,j}\|R_{ij}\|_{\psi_2} \leq K$ (see Appendix~\ref{app:dist} for a definition of the sub-gaussian norm $\|\cdot\|_{\psi_2}$), the {\sl product} matrix $MR$ satisfies $\|MR\bu\|_2 = (1\pm \eps) \|M\|_{\mathsf{F}}\|\bu\|_2$, whenever $\bu \in \Sigma_s$, with failure probability at most $\exp\left(- c\eps^2\cdot\mathrm{sr}(M)/K^4 \right)$, for some absolute constant $c >0$, provided that the {\em stable rank} of $M$, denoted $\mathrm{sr}(M)$, satisfies 
$$
\mathrm{sr}(M) := \frac{\|M\|_{\mathsf{F}}^2}{\|M\|_2^2} =  \Omega \left(\frac{K^4}{\eps^2} s \log\left(\frac{d}{s}\right) \right).
$$
Requiring that $\|MR\bu\|_2$ be tightly concentrated around $\|M\|_{\mathsf{F}}\|\bu\|_2$ is unavoidable, for indeed $\Ex \left\{\|MR\bu\|_2 \right\} =  \|M\|_{\mathsf{F}}\|\bu\|_2$ (see~\cite{KR19} for details); this compels the normalisation $M/\|M\|_{\mathsf{F}}$ if one is to reach the approximate isometry required by RIP~\eqref{eq:RIP}.

The two perturbation models introduced thus far, namely the {\sl additive} one $M+R$ and the {\sl multiplicative} one $MR$, are of interest in this venue. Pursuing RIP in the additive model, however, is futile regardless of the random perturbation chosen. To see this, let $\eps >0$ and $M \in \R^{m \times n}$ be fixed, and consider for concreteness a sub-gaussian random matrix $R \in \R^{m \times n}$ whose entries are centred and have variance one (though, any type of random matrix will do). Then, for any $\bx \in \R^n$, we may write 
\begin{equation}\label{eq:additive-RIP}
\left\|\left(M+\frac{\eps}{\sqrt{n}}R\right)\bx\right\|_2^2 = \|M\bx\|_2^2 + \frac{2\eps}{\sqrt{n}} \sum_i \inner{\bm_i}{\bx}\inner{\br_i}{\bx} + \frac{\eps^2}{n}\|R\bx\|_2^2,
\end{equation}
where $\bm_i$ and $\br_i$ denote the $i$th rows of $M$ and $R$, respectively. In order to compel~\eqref{eq:additive-RIP} to be close to $\|\bx\|_2^2$, as to ensure an approximate isometry, the distortion of $R$ is controlled through the normalisation by $n^{-1/2}$; to see this, consult, e.g.,~\cite[Theorem~4.6.1]{Vershynin} providing tight two-sided estimates for the extreme singular values of $R$. Alas, the deterministic term $\|M\bx\|_2$ present in~\eqref{eq:additive-RIP} mandates that for an approximate isometry to be obtained post perturbation by $R$, the original seed matrix $M$ must be an approximate isometry itself defeating the caveated message of the random perturbation ``mending" $M$ into an approximate isometry. 
For NSPs, however, our results show that the additive perturbation model does make sense.

\subsection{Our results}\label{sec:results} In this section, we state our main result, namely Theorem~\ref{thm::main-perturb}. Appreciation for the more abstract formulation of the latter, we develop by first presenting its implications pertaining to the $\ell_2$-RNSP property of fixed matrices randomly perturbed (in an additive fashion) using random matrices whose entries obey classical, so to speak,  distributional laws, namely sub-gaussian and sub-exponential. This can be seen in Corollaries~\ref{cor:classical} and~\ref{cor:classical-perturb}.  Theorem~\ref{thm::main-perturb} ventures much further than these restricted corollaries thereof by accommodating random matrix perturbations, entries of which are heavy-tailed in the sense of~\eqref{eq:beyond-exp}.       



Outline of the core approach used to prove the aforementioned results is provided in Section~\ref{sec:approach} whilst in Section~\ref{sec:contribution}, the impact of our results is detailed.  

\begin{remark}
We do not pursue $(\ell_q,\ell_p)$-RNSP results for (additively) randomly perturbed matrices as such results can be established through an adaptation of our arguments here supporting $\ell_2$-RNSP of (additively) randomly perturbed matrices. 
\end{remark}

\medskip
\noindent
{\bf Notation.} Our notation follows that of~\cite{Vershynin}. Given a random variable $X$ and a positive integer $p$, write $\mu_p(X) := \Ex \{(X -\Ex X)^p \}$ to denote the $p$-{\em central moment} of $X$. Note that $\mu_1(X) = 0$ always holds and $\mu_2(X)$ is simply the variance of $X$.    
A random matrix $A = (a_{ij}) \in \R^{m\times n}$ is said to be $k$-{\em central} if $\mu_p(a_{ij}) = \mu_p(a_{st})$ holds for every $p \in [k]$ and whenever $i,s \in [m]$ and  $j,t \in [n]$. 
Set 
$$
\mu_k(A):= \max_{\substack{ i\in [m] \\ j \in [n]}} |\mu_k(a_{ij})|.
$$
As noted above, $\mu_1(A) = 0$ always holds. For instance, $2$-centrality of a matrix simply means that all entries of the latter have the same variance. A reoccurring quantity in our results below is
$$
K(A) := \mu_2(A)^2+\mu_3(A)+ \mu_2(A)\mu_3(A) + \mu_4(A);
$$
which is well-defined provided the involved central moments are all finite. 

\medskip
\noindent
Similar notation, to that set here for matrices, applies to vectors as well.


\bigskip
\noindent
{\bf Classical perturbations.} For a random matrix $A = (a_{ij}) \in \R^{m \times n}$ and $k \in \{1,2\}$, set
$$
\|A\|_{\psi_k} := \max_{\substack{i \in [m]\\j \in [n]}} \|a_{ij}\|_{\psi_k};
$$
a definition of the so-called sub-exponential and sub-gaussian norms $\|\cdot\|_{\psi_1}$ and $\|\cdot\|_{\psi_2}$, respectively, can be seen in Appendix~\ref{app:dist}. The first corollary derived from our main result, namely Theorem~\ref{thm::main-perturb}, reads as follows. 

\begin{corollary}\label{cor:classical}
Let $k \in \{1,2\}$ and let $A \in \R^{m \times n}$ be a $2$-central random matrix. Then, the matrix $A$ satisfies the $\ell_2$-$\mathrm{RNSP}(s,\rho,\tau)$ property with probability at least $1-\exp\left(-\frac{\mu_2(A)^4}{K(A)^2} \cdot m \right)$, whenever $\rho \in (0,1)$, $\tau >0$, $s \in [n]$, and 
\begin{equation}\label{eq:m1}
m \geq C \left(\frac{K(A)}{\mu_2(A)^{5/2}\tau}\left(1+ \tau \rho^{-1}\|A\|_{\psi_k} \sqrt{s \log n} \right)\right)^2,
\end{equation}
where $C > 0$ is some sufficiently large absolute constant. 
\end{corollary}

\begin{remark}
If the terms $\rho,\tau,\|A\|_{\psi_k}$, as well as $(\mu_i)_{i\in [4]}$, appearing in~\eqref{eq:m1}, are all independent of $s$ and $n$, and $s \leq n^{1-\eps}$ holds for some arbitrarily small yet fixed $\eps >0$, then the lower bound~\eqref{eq:m1} asymptotically coincides with the optimal lower bound~\eqref{eq:optimal-m}.    
\end{remark}

\begin{remark}
We do not pursue the optimality of the quotient $\frac{K(A)}{\mu_2(A)^{5/2}}$.    
\end{remark}

Unlike previous results~\cite{SGR15,KJ16,KJ16b,LM17,SJC19,SJC19b}, Corollary~\ref{cor:classical} does not impose a strict distributional identity on either the entries of the matrix or its rows. This, in turn, allows for the following smoothed analysis type result with respect to $\ell_2$-RNSP. 

\begin{corollary}\label{cor:classical-perturb}
Let $k \in \{1,2\}$, let $M \in \R^{m\times n}$ be an arbitrary yet fixed matrix, and let $R \in \R^{m \times n}$ be a $2$-central random matrix. Then, $M+R$ satisfies the $\ell_2$-$\mathrm{RNSP}(s,\rho,\tau)$ property with probability at least $1-\exp\left(-\frac{\mu_2(R)^4 }{K(R)^2} \cdot m\right)$, whenever $\rho \in (0,1)$, $\tau >0$, $s \in [n]$, and 
\begin{equation}\label{eq:m-R}
m \geq C \left(\frac{K(R)}{\mu_2(R)^{5/2}\tau}\left(1+ \tau \rho^{-1}\sqrt{s \log n} \left(\|R\|_{\psi_k}+ \|M\|_\infty\right)\right)\right)^2,
\end{equation}
where $C > 0$ is some sufficiently large absolute constant. 
\end{corollary}

\begin{remark}
To deduce Corollary~\ref{cor:classical-perturb} from Corollary~\ref{cor:classical}, set $A:= M+R$ and note that the $\ell$-central moments of the entries of $M+R$ obey
\begin{equation}\label{eq:central-moments}
\Ex \left\{ \left(M_{ij}+R_{ij} - \Ex \left\{M_{ij}+R_{ij} \right\} \right)^\ell \right\} = \Ex \left\{ (R_{ij} -\Ex R_{ij})^\ell \right\} 
\end{equation}
and thus coincide with those of $R$. Hence, applying Corollary~\ref{cor:classical} with $A := M + R$, coupled with the fact that  
$$
\|A\|_{\psi_k} = \|M +R\|_{\psi_k} \leq \|M\|_{\psi_k} + \|R\|_{\psi_k} \overset{\eqref{eq:sub-gauss-bounded}}{=} O(\|M\|_\infty) + \|R\|_{\psi_k}
$$
holds, whenever $k \in \{1,2\}$, delivers Corollary~\ref{cor:classical-perturb}. 
\end{remark}

If $\rho,\tau,\|R\|_{\psi_k}, \|M\|_\infty$ as well as $(\mu_i)_{i\in [4]}$, appearing in~\eqref{eq:m-R}, are all independent of $s$ and $n$, and $s \leq n^{1-\eps}$ holds for some arbitrarily small yet fixed $\eps >0$, then Corollary~\ref{cor:classical-perturb} carries the message that {\sl any} matrix $M$ (with $\|M\|_\infty$ as above) can be randomly (additively) perturbed (per entry) by a sub-gaussian or a sub-exponential matrix as to a.a.s. result in a matrix satisfying $\ell_2$-RNSP whilst asymptotically  realising the optimal bound~\eqref{eq:optimal-m}.

\bigskip
\noindent
{\bf Heavy-tailed perturbations.} Our main result, namely Theorem~\ref{thm::main-perturb}, is stated next. To that end, and following in the wake of~\eqref{eq:beyond-exp}, given $\kappa,\alpha >0$, as well as a positive integer $r$, all independent of one another, we refer to a random variable $X$ as being $(\kappa,\alpha,r)$-{\em reasonable} provided $\|X\|_p \leq \kappa p^{\alpha}$ 
holds, whenever $p \leq r$; if all moments of $X$ are so bounded (i.e., $r = \infty$), then the latter is referred to as $(\kappa,\alpha)$-reasonable. A matrix is said to be $(\kappa,\alpha,r)$-reasonable (respectively, $(\kappa,\alpha)$-reasonable) if each of its entries is $(\kappa,\alpha,r)$-reasonable (respectively, $(\kappa,\alpha)$-reasonable). Note that sub-gaussian random variables $X$ are $(O(\|X\|_{\psi_2}),1/2)$-reasonable and sub-exponential ones are $(O(\|X\|_{\psi_1}),1)$-reasonable; rendering the family of $(\kappa,\alpha,r)$-reasonable distributions as fairly wide and one which accommodates some rather heavy-tailed distributions. 

Given $A = (a_{ij}) \in \R^{m \times n}$ as well as $s \in [n]$, let
\begin{equation}\label{eq:local-frob}
\|A\|_{\frob,s} := \max_{\substack{S \subseteq [n]\\|S|=s}} \sqrt{\sum_{j \in S} \sum_{i=1}^m a_{ij}^2}
\end{equation}
denote the largest Frobenius norm witnessed by an $m \times s$ submatrix of $A$. 

\medskip
Our main result reads as follows. 

\begin{theorem}\label{thm::main-perturb}
Let $A \in \R^{m \times n}$ be $2$-central and such that $A - \Ex\{A\}$ is $(\kappa,\alpha,\log (en))$-reasonable for some $\kappa >0$ and $\alpha \geq 1/2$. Then, $A$ satisfies the $\ell_2$-$\mathrm{RNSP}(s,\rho,\tau)$ property with probability at least $1-\exp\left(-\Omega\left(\frac{\mu_2(A)^4 }{K(A)^2}\cdot m\right) \right)$, whenever $\rho \in (0,1)$, $\tau >0$, $s \in [n]$, and 
\begin{equation}\label{eq:m}
m \geq \max \left\{  C \left(\frac{K(A)}{\mu_2(A)^{5/2}\tau}\left(1+ \tau \rho^{-1} \Big(e^{2\alpha}\kappa \sqrt{s \log (en/s)} + F(A) \Big)\right)\right)^2, \left(\log (en)\right)^{\max\{2\alpha-1,1\}}\right\},
\end{equation}
where $F(A) := \min\{\|\Ex\{A\}\|_{\frob,s}, \sqrt{s\log n} \|\Ex\{A\}\|_{\infty}\}$, and $C > 0$ is some sufficiently large constant.
\end{theorem}

\begin{remark}
In Theorem~\ref{thm::main-perturb}, reasonability is imposed on $A - \Ex\{A\}$ in order to facilitate a centring argument employed in its proof; one might as well impose reasonability on $A$ alone and mildly adapt our proof to fit this alternative. 
\end{remark}

Reasonable distributions for which $\alpha > 1$ holds are rather heavy-tailed; for these, the maximum seen on the right hand side of~\eqref{eq:m} is determined by its second term and thus provide a quantitive measure through which asymptotic departure from the optimal bound~\eqref{eq:optimal-m} for $m$ can be gauged. If, however, $1/2 \leq \alpha \leq 1$ holds, then the optimal bound~\eqref{eq:optimal-m} for $m$ is asymptotically met provided that $F(A) = O\left(\sqrt{s\log(en/s)}\right)$ and all other terms involved are independent of $s$ and $n$. 

\begin{remark}
To deduce Corollary~\ref{cor:classical} from Theorem~\ref{thm::main-perturb}, recall first that $\Ex\{|X|\} = O(\|X\|_{\psi_k})$ holds for any random variable whenever $k \in \{1,2\}$; then, proceed to set $\alpha = 1/2$, to retrieve the sub-gaussian setting, and $\alpha =1$, for the sub-exponential case. In these settings, $\kappa$ coincides with $\|A\|_{\psi_2}$ in the sub-gaussian case and with $\|A\|_{\psi_1}$ in the sub-exponential one.
\end{remark}

\begin{remark}\label{rem:two-alt}
The two alternatives available through $F(A)$ are of different natures; the first imposes a more global condition on the matrix to satisfy whilst the second has a clear local flavour.       
For the sake of brevity, Corollary~\ref{cor:classical} does not utilise the full strength, so to speak, of Theorem~\ref{thm::main-perturb}; from the two alternatives available through $F(A)$, Corollary~\ref{cor:classical} retains the local alternative only. In the context of Smoothed Analysis, however, both alternatives have merit as each imposes rather different restrictions on the deterministic seed matrix.
\end{remark}




\begin{corollary}\label{cor::main-perturb}
Let $M \in \R^{m\times n}$ be an arbitrary yet fixed matrix and let $R \in \R^{m \times n}$ be $2$-central and such that $R - \Ex\{R\}$ is $(\kappa,\alpha,\log (en))$-reasonable for some $\kappa >0$ and $\alpha \geq 1/2$. Then, $M+R$ satisfies the $\ell_2$-$\mathrm{RNSP}(s,\rho,\tau)$ property with probability at least $1-\exp\left(-\Omega\left(\frac{\mu_2(R)^4 }{K(R)^2} \cdot m\right)\right)$, whenever $\rho \in (0,1)$, $\tau >0$, $s \in [n]$, and provided that $m \geq \max\{\ell_1,\ell_2\}$, where 
$$
\ell_1 := C \left(\frac{K(R)}{\mu_2(R)^{5/2}\tau}\left(1+ \tau \rho^{-1} \Big(e^{2\alpha}\kappa \sqrt{s \log (en/s)} + F(M+R)\Big)\right)\right)^2,
$$
where $C>0$ is a sufficiently large constant; and where
$$
\ell_2:= \left(\log (en)\right)^{\max\{2\alpha-1,1\}}.
$$
\end{corollary}

\begin{remark}
To deduce Corollary~\ref{cor::main-perturb} from Theorem~\ref{thm::main-perturb}, note that in addition to~\eqref{eq:central-moments}, and by the same token, 
$$
\|M+R - \Ex\{M+R\}\|_p = \|R - E\{R\}\|_p
$$
holds as well. Corollary~\ref{cor::main-perturb} then follows by applying Theorem~\ref{thm::main-perturb} to $A:= M+R$. 
\end{remark}

\begin{remark}
In the case of the so-called classical perturbations, seen in Corollary~\ref{cor:classical-perturb}, an emphasis is put solely on $\|M\|_\infty$; this is done for the sake brevity only. For indeed, Theorem~\ref{thm::main-perturb} allows for two alternatives to be imposed on the seed matrix; the second is captured through the localised Frobenius norm $\|M\|_{\frob,s}$.   
\end{remark}

\subsubsection{Overarching argument}\label{sec:approach} Theorem~\ref{thm::main-perturb} is proved through the same high-level conceptual approach utilised in~\cite{SGR15,KJ16b,KJ16,SJC19,SJC19b}; this is detailed next. 
For an integer $s \in [n]$ and a vector $\bv \in \R^n$, write $S_{\max} := S_{\max}(\bv,s) \subseteq [n]$ to denote the set of $s$ indices holding the $s$ largest entries of $\bv$ in absolute value; set $S_{\min} := [n] \sm S_{\max}$. A matrix $A \in \R^{m \times n}$ satisfies $\ell_2-\mathrm{RNSP}(s,\rho,\tau)$ provided that
\begin{equation}\label{eq:rnsp-2}
\| \bv_{S_{\max}}\|_2 \overset{\eqref{eq:rnsp}}{\leq} \frac{\rho}{\sqrt{s}} \|\bv_{S_{\min}}\|_1 + \tau \|A\bv\|_2 
\end{equation}
holds for every $\bv \in \R^n$. This inequality, remaining invariant under rescaling of $\bv$, allows one to assume that $\|\bv\|_2 = 1$. Any vector $\bv \in \R^n$ satisfying 
\begin{equation}\label{eq:Tps}
\bv \notin T_{\rho,s} := \left\{\bu \in \R^n: \|\bu\|_2 =1, \|\bu_{S_{\max}}\|_2 > \frac{\rho}{\sqrt{s}} \|\bu_{S_{\min}}\|_1\right\}
\end{equation}
satisfies~\eqref{eq:rnsp-2}. It follows that $A$ satisfies $\ell_2-\mathrm{RNSP}(s,\rho,\tau)$ whenever
\begin{equation}\label{eq:rnsp-goal}
\inf\left\{ \|A\bv\|_2: \bv \in T_{\rho,s} \right\} > \tau^{-1}
\end{equation} 
holds. 

To establish Theorem~\ref{thm::main-perturb}, we prove that the matrix $A$, per that theorem, satisfies~\eqref{eq:rnsp-goal} with the probability stated in said theorem, using (a slight adaptation of) {\sl Mendelson's method}, namely Theorem~\ref{thm::Mendelson}. This method is employed in~\cite{SGR15,KJ16b,KJ16,SJC19,SJC19b}. Our implementation of this method differs substantially from these previous instantiations; details follow in the next section. 


\subsubsection{Our contribution}\label{sec:contribution} 

\medskip
\noindent
{\bf Smoothed Analysis for RNSPs.} As far as we could ascertain, our results presented above are the first to introduce the influential framework of Smoothed Analysis to the rich realm of Compressed Sensing in a systematic manner. In the vein of Smoothed Analysis, in order to realise the optimality seen in~\eqref{eq:optimal-m}, the sole restriction (implicitly) imposed on the seed matrix $M$, in Corollary~\ref{cor:classical-perturb}
is that $\|M\|_\infty$ must be independent of the dimensions of $M$ (i.e., be dimension-free) as well as from $s$ - the sparsity of the signal being reconstructed; the impositions placed on the seed matrix in Corollary~\ref{cor::main-perturb} are discussed in Remark~\ref{rem:two-alt}. Collectively, the restrictions imposed on the seed matrix encountered throughout our results, are significantly milder compared to those seen in~\cite{KR19,SJC19,SJC19b}. Indeed, in~\cite{SJC19,SJC19b}, all entries of the seed matrix $M$ are mandated to be equal to the {\sl same} constant; in~\cite{KR19} a stable-rank assumption is imposed. 


 Throughout our results, the seed matrix $M$ is allowed to retain a rather untamed and fairly arbitrary nature. The ubiquity claims made for NSPs, throughout the rather intensive aforementioned line of research dedicated to the study of said properties in truly random matrices, is then enhanced carrying the message that rather arbitrary fixed matrices (subject to requirements mentioned above) can be ``mended", through additive random perturbations, as to yield matrices satisfying a targeted NSP.

\medskip
\noindent
{\bf Comparison with previous results.}  With the exception of imposing the same variance across all entries of the matrix, our results, namely Theorem~\ref{thm::main-perturb} and corollaries thereof, strip all other distributional-similarity requirements present in previous results~\cite{SGR15,KJ16b,KJ16,LM17,SJC19,SJC19b}, where identical distributions are imposed on the entries/rows of the matrix. This departure from (strict) distributional-similarity has two crucial effects; the first is conceptual in the sense that it allows for the systematic introduction of Smoothed Analysis into Compressed Sensing, the second manifests itself in proof techniques. For instance, we manage (or compelled) to avoid appeals to Dudley's integral identity (see, e.g.,~\cite[Theorem~8.1]{Vershynin}) which is central in~\cite{SGR15,KJ16b,KJ16,SJC19,SJC19b}; we thus avoid any use of $\eps$-nets in our arguments.  

The sub-gaussian setting of Corollary~\ref{cor:classical} captures the main result of~\cite{KJ16b,KJ16}, dedicated to Bernoulli matrices. The results of~\cite{SJC19,SJC19b}, pertaining to rather restrictive sub-gaussian additive perturbation, allow for dependencies within each single row whilst Corollary~\ref{cor:classical-perturb} does not accommodate this. Every other feature in the results of~\cite{SJC19,SJC19b} is, however, captured and generalised in the sub-gaussian setting of said corollary. The sub-exponential setting of Corollary~\ref{cor:classical} captures and generalises the aforementioned results appearing in~\cite{SGR15,LM17} and provides Smoothed Analysis variants thereof.

As mentioned in Section~\ref{sec:approach}, the application of Mendelson's method through which the overarching argument seen in said section is executed is not new and has been utilised in~\cite{SGR15,KJ16b,KJ16,LM17,SJC19,SJC19b}. Nevertheless, the implementation details underpinning our application of Mendelson's method are the driving force leading to various gains compared to previous results. Roughly put, yet thoroughly explained in Section~\ref{sec:Mendelson} below, Mendelson's method is comprised of two core ingredients; the first is a {\sl small-ball probability} estimation, captured through~\eqref{eq:Q}, and the second is a measure for the so-called {\sl empirical width} of a certain ``random walk", captured through~\eqref{eq:W}, performed using the rows of the matrix along the set $T_{\rho,s}$, defined in~\eqref{eq:Tps}. In our application of Mendelson's method, we provide a {\sl distribution-free} small-ball probability estimate, seen in Proposition~\ref{prop:small-ball}, requiring only that the variance be identical across all entries and that their fourth moments be finite. It is our empirical width estimations that require tail-information. This conceptual separation is absent in~\cite{SGR15,KJ16b,KJ16,SJC19,SJC19b}. In~\cite[Theorem~A]{LM17}, a parametric lower bound on the small-ball probability required by the method is part of the assumptions of the theorem. In~\cite{SGR15,KJ16b,KJ16,SJC19,SJC19b}, tail properties are employed in both parts of Mendelson's method. 


\section{Mendelson's method}\label{sec:Mendelson}

In this section, we state the so-called {\em Mendelson's method}~\cite{M15}; see~\cite[Section~5]{T15} for a broader insightful exposition. Casually stated, the crux of Mendelson's method is that it provides effective lower bounds for non-negative empirical processes, without the need for the involved distributions to exhibit strict concentration to expectation properties, allowing for the study of heavy-tailed distributions across a wide range of applications. 
 
Given $\eps > 0$ and random vectors $\bphi_1,\ldots,\bphi_m \in \mathbb{R}^n$ as well as a subset $E \subset \R^n$, let
\begin{equation}\label{eq:Q}
Q_\eps(E;\bphi_1,\ldots,\bphi_m) := \inf_{\substack{\bu \in E\\ i \in [m]}} \Pr\left\{ |\langle \bphi_i,\bu \rangle | \geq \eps \right\}. 
\end{equation}
The next quantity involved in said method is referred to as the {\em empirical width} of the set $E$ introduced above. In that, given independent Rademacher random variables $\xi_1,\ldots,\xi_m$ that are independent of the vectors $\bphi_1,\ldots,\bphi_m$, set
\begin{equation}\label{eq:W}
W(E;\bphi_1,\ldots,\bphi_m) := \Ex \left\{\sup_{\bu \in E} |\langle \bh, \bu \rangle| \right\},
\end{equation}
where 
\begin{equation}\label{eq:h}
\bh := \frac{1}{\sqrt{m}} \sum_{i=1}^m \xi_i \bphi_i. 
\end{equation}

The following result, namely Theorem~\ref{thm::Mendelson}, is a slight extension of Mendelson's result~\cite[Theorem~5.4]{M15}, where we follow the formulation seen in~\cite[Proposition~5.1]{T15}. The nature of our extension is divulged in Remark~\ref{rem::Mendelson-ext} below; the proof of Theorem~\ref{thm::Mendelson} is postponed until Appendix~\ref{app:Mendelson}.   
 
\begin{theorem}\label{thm::Mendelson}
Let $E \subset \R^n$ be fixed and let $\bphi_1,\ldots,\bphi_m \in \R^n$ be independent random vectors. Then, for any $\eps >0$ and $t >0$,
$$
\inf_{\bu \in E} \left( \sum_{i=1}^m \langle \bphi_i, \bu \rangle^2 \right)^{1/2} \geq \eps \sqrt{m} Q_{2\eps}(E;\bphi_1,\ldots,\bphi_m) - 2 W(E;\bphi_1,\ldots,\bphi_m ) - \eps t
$$
holds with probability at least $1-\exp(-t^2/2)$.
\end{theorem}

\begin{remark}\label{rem::Mendelson-ext}
In the original result of Mendelson~\cite[Theorem~5.4]{M15}, the vectors $\bphi_1,\ldots,\bphi_m$ are mandated to be identically distributed. This restriction does not fit our needs and the removal of this requirement in Theorem~\ref{thm::Mendelson} constitutes our adaptation of said method. 
\end{remark}

\section{Proof of the main result: Theorem~\ref{thm::main-perturb}}

\subsection{Core ingredients}  In this section, we state Propositions~\ref{prop:small-ball} and~\ref{prop::W}, which constitute the core ingredients for our proof of Theorem~\ref{thm::main-perturb}. These two propositions constitute our implementation of Mendelson's method, namely Theorem~\ref{thm::Mendelson}. Proposition~\ref{prop:small-ball} handles the small-ball probability part involved in said method; the latter being distribution-free is the sole reason for us being able to apply Mendelson's method for a wide range of distributions. Proposition~\ref{prop::W} handles the empirical width aspect of said method. 

\medskip

The following is our aforementioend distribution-free small-ball probability result. A random variable $X$ satisfying $|\mu_p(X)| < \infty$ is said to be $p$-{\em bounded}; a simple application of H\"older's inequality\footnote{Let $\bx,\by \in \R^n$ as well as $p,q \in [1,\infty]$ satisfying $p^{-1} + q^{-1}=1$ be given. Then, $|\inner{\bx}{\by}| \leq \|\bx\|_p\|\by\|_q$.} delivers that if $X$ is $q$-bounded, then it is also $p$-bounded for every $1\leq p < q$. Sub-gaussian, sub-exponential, and $(\kappa,\alpha,r)$-reasonable distributions, with $r \geq 4$, are all $4$-bounded.
If all entries of a random vector/matrix are $p$-bounded, then the latter is said to be $p$-{\em bounded} as well. 

\begin{proposition}\label{prop:small-ball}
Let $A \in \R^{m \times n}$ be a $2$-central and $4$-bounded random matrix with independent entries. Then, for every $\rho \in (0,1)$ and any $s \in [n]$, the equality 
$$
Q_{2\eps}\left(T_{\rho,s}; \ba_1,\ldots,\ba_m \right) = \Omega \left( \frac{\mu_2(A)^2}{K(A)}\right)
$$
holds, where $(\ba_i)_{i \in [m]}$ denote the rows of $A$, $\eps := \sqrt{\mu_2(A)}/4$, $T_{\rho,s}$ is as per~\eqref{eq:Tps}, and where $Q_{2\eps}$ is as defined in~\eqref{eq:Q}.  
\end{proposition}


The next result pertains to the empirical width component in Mendelson's method fitting the setting of Theorem~\ref{thm::main-perturb}. 

\begin{proposition}\label{prop::W}
Let $\ba_1,\ldots,\ba_m \in \R^n$ be the rows of a random matrix $A \in \R^{m \times n}$ whose entries are pairwise independent and for which $A - \Ex\{A\}$ is 
$(\kappa,\alpha,\log (en))$-reasonable for some $\kappa >0$ and $\alpha \geq 1/2$. Then,
$$
W(T_{\rho,s};\ba_1,\ldots,\ba_m) = O\left(\rho^{-1} e^{2\alpha} \kappa 
\sqrt{s \log \left(\frac{en}{s} \right)}+\rho^{-1}F(A)\right),
$$ 
where $F(A) = \min\{\|\Ex\{A\}\|_{\frob,s}, \sqrt{s\log n} \|\Ex\{A\}\|_{\infty}\}$, and whenever $s \in [n]$,$m \geq \left( \log(en) \right)^{\max\{2\alpha - 1, 1 \}}$,  $\rho \in (0,1)$ and $T_{\rho,s}$ is as per~\eqref{eq:Tps}.
\end{proposition}


Postponing the proof of Proposition~\ref{prop:small-ball} until Section~\ref{sec:small-ball}, and that of Proposition~\ref{prop::W} until Section~\ref{sec:Rademacher}, we proceed, first, with the proof of Theorem~\ref{thm::main-perturb} assuming the former two propositions both hold true.

\subsection{Proof of Theorem~\ref{thm::main-perturb}} In this section, we prove Theorem~\ref{thm::main-perturb} under the assumption that Propositions~\ref{prop:small-ball} and~\ref{prop::W} both hold true. To that end, it suffices to verify that, in the setting of Theorem~\ref{thm::main-perturb}, inequality~\eqref{eq:rnsp-goal} holds with the probability stipulated in said Theorem. Mendelson's method, namely Theorem~\ref{thm::Mendelson}, applied with $\eps := \sqrt{\mu_2(A)}/4$, and coupled with Propositions~\ref{prop:small-ball} and~\ref{prop::W}, collectively assert that 
$$
\inf_{\bu \in T_{\rho,s}} \|A\bu\|_2 \geq \Omega\left( \frac{\mu_2^{5/2}(A)}{K(A)}\sqrt{m}\right) - O \left(\rho^{-1} \left(e^{2\alpha} \kappa \sqrt{s \log \left(\frac{en}{s} \right)} + F(A)\right) \right) - \sqrt{\mu_2(A)}t/4
$$ 
holds for all $t\geq 0$ with probability at least $1 - \exp(-t^2/2)$. Setting $t = c \frac{\mu_2(A)^2}{K(A)} \sqrt{m}$, where $c > 0$ is a sufficiently small constant, yields 
\begin{equation}\label{eq:mendelson-gauss}
\inf_{\bu \in T_{\rho,s}} \|A\bu\|_2 = \Omega\left( \frac{\mu_2^{5/2}(A)}{K(A)}\sqrt{m}\right) - O \left(\rho^{-1} \left(e^{2\alpha} \kappa \sqrt{s \log \left(\frac{en}{s} \right)}  +F(A)\right) \right)
\end{equation}
holds with probability at least
$$
1- \exp \left(- \Omega \left(\frac{\mu_2(A)^4}{K(A)^2} m \right) \right).
$$
Then, for
$$
m \geq \frac{CK(A)^2\left(1 + \tau \rho^{-1} \left(e^{2\alpha} \kappa \sqrt{s \log \left(\frac{en}{s} \right)} +F(A) \right)\right)^2}{\mu_2^5(A) \tau^2},
$$
where $C$ is a sufficiently large constant, the first term appearing on the right hand side of~\eqref{eq:mendelson-gauss} strictly exceeds the second term appearing on the same side of~\eqref{eq:mendelson-gauss} by an additive factor of at least $\tau^{-1}$ leading to the verification of~\eqref{eq:rnsp-goal} with the aforementioned probability. Theorem~\ref{thm::main-perturb} then follows. \hfill\QED

\section{Distribution-free small-ball probabilities}\label{sec:small-ball} In this section, we prove Proposition~\ref{prop:small-ball}. The following two results facilitate our proof of the latter. 


\begin{observation}\label{lem:inner-prod}
Let $\bx \in \R^n$ be $2$-central with independent entries. Then, 
$$
\Ex \left\{ \langle\bx,\bz \rangle^2 \right\} = \mu_2(\bx)\|\bz\|_2^2 + \langle\Ex\{\bx\},\bz \rangle^2 
$$
holds, whenever $\bz \in \R^n$. 
\end{observation}

\begin{proof}
Set $\bmu := \Ex \{\bx\}$ and let $\br := \bx - \bmu$.     Then,
\begin{align*}
\Ex \left\{\inner{\bx}{\bz}^2\right\} & = \Ex \left\{\inner{\br+\bmu} {\bz}^2\right\} \\
& = \Ex \left\{\left( \inner{\br}{\bz} + \inner{\bmu}{\bz}\right)^2 \right\}\\
        & = \Ex \left\{ \inner{\br}{\bz}^2\right\} + 2\Ex\left\{\inner{\br}{\bz}\inner{\bmu}{\bz}\right\} + \inner{\bmu}{\bz}^2\\
        & = \Ex \left\{ \inner{\br}{\bz}^2\right\} +  \inner{\bmu}{\bz}^2,
\end{align*}
where for the last equality we rely on the entries of $\br$ being centred.  
Proceed to write 
\begin{align*}
\Ex \left\{ \inner{\br}{\bz}^2\right\} &= \Ex \left\{ \sum_{i,j \in [n]} \br_i \br_j\bz_i\bz_j \right\}\\
& = \sum_{i=1}^n \Ex\{\br_i^2\} \bz_i^2 + \sum_{\substack{i,j \in [n]\\ i\neq j}} \Ex\{\br_i\} \Ex\{\br_j\} \bz_i \bz_j \\
& = \mu_2(\bx) \|\bz\|_2^2,
\end{align*}
where in the second equality we rely on the entries of $\br$ being pairwise independent, and the last equality is owing to the entries of $\br$ being centred and to $\Ex\{\br_i^2\} = \mu_2(\bx)$ holding for every $i \in [n]$ on account of $\bx$ being $2$-central. The claim follows.
\end{proof} 

For a $4$-bounded $\bx \in \R^n$, set
$$
\tilde K(\bx) := 3 \mu_2(\bx)^2 + 4 \mu_3(\bx) +2 \mu_2(\bx)\mu_3(\bx) + \mu_4(\bx);
$$
while $\tilde K(\bx)$ and $K(\bx)$ do not coincide, they do have the same order of magnitude.

\begin{lemma}\label{lem:deviation}
Let $\bx \in \R^n$ be 
$2$-central and $4$-bounded
with independent entries. Then, 
$$
\Pr \Big\{ \langle \bx,\bz \rangle^2 \geq \theta \Ex \left\{ \langle \bx,\bz \rangle^2 \right\} \Big\} \geq (1-\theta)^2 \frac{\mu_2(\bx)^2}{\tilde K(\bx)}
$$
holds, whenever $\bz \in \mathbb{S}^{n-1}$ and $\theta \in [0,1]$. 
\end{lemma}

Postponing the proof of Lemma~\ref{lem:deviation} until Section~\ref{sec:deviation}, we are in a position to prove Proposition~\ref{prop:small-ball}.

\begin{proofof}{Proposition~\ref{prop:small-ball}} Recalling that $\eps = \sqrt{\mu_2(A)}/4$, write 
\begin{align*}
Q_{2\eps}(T_{\rho,s}; \ba_1,\ldots,\ba_m) & = \inf_{\substack{\bu \in T_{\rho,s}\\ i \in [m]}} \Pr\left\{|\langle \ba_i,\bu \rangle| \geq  \sqrt{\mu_2(A)}/2\right\}\\
& = \inf_{\substack{\bu \in T_{\rho,s}\\ i \in [m]}} \Pr\left\{\langle \ba_i,\bu \rangle^2 \geq  \mu_2(A)/4\right\}.
\end{align*}
Since $\Ex\left\{\langle \ba_i,\bu \rangle^2 \right\} \geq \mu_2(A)$ holds for every $\bu \in T_{\rho,s}$ by Observation~\ref{lem:inner-prod}, we may proceed to write 
$$
Q_{2\eps}(T_{\rho,s},\ba_1,\ldots,\ba_m) \geq 
\inf_{\substack{\bu \in T_{\rho,s}\\ i \in [m]}} \Pr\left\{\langle \ba_i,\bu \rangle^2 \geq  \Ex\left\{\langle \ba_i,\bu \rangle^2 \right\}/4\right\}. 
$$
Lemma~\ref{lem:deviation}, applied with $\theta = 1/4$, then asserts that for every $\bu \in T_{\rho,s}$ and every $i \in [m]$ we have  
\begin{align*}
\Pr\Big\{\langle \ba_i,\bu \rangle^2 \geq  \Ex\left\{\langle \ba_i,\bu \rangle^2 \right\}/4\Big\} = \Omega \left(\frac{\mu_2(\ba_i)^2}{\tilde K(\ba_i)}\right) = \Omega \left( \frac{\mu_2(A)^2}{K(A)}\right), 
\end{align*}
where in the last equality we use the previously noted fact that $\tilde K(\ba_i) = \Theta(K(A))$.
\end{proofof}

\subsection{Proof of Lemma~\ref{lem:deviation}}\label{sec:deviation} The first result facilitating our proof of Lemma~\ref{lem:deviation} is the so-called Paley-Zygmund inequality asserting that 
$$
\Pr \Big\{ Z \geq \theta \Ex \{Z\} \Big\} \geq (1-\theta)^2 \frac{\Ex \{Z\}^2}{\Ex \{Z^2\}}
$$
holds, whenever $Z$ is a non-negative random variable with finite variance and $\theta \in [0,1]$. The second tool reads as follows. 

\begin{lemma}\label{lem:4-moment}
Let $\bx \in \R^n$ be $2$-central and $4$-bounded
and let $\bz \in \mathbb{S}^{n-1}$. Then, 
$$
\Ex \left\{\inner{\bx}{\bz}^4 \right\} \leq \frac{\tilde K(\bx)}{\mu_2(\bx)^2} \cdot \Ex \left\{ \inner{\bx}{\bz}^2 \right\}^2.
$$
\end{lemma}

With the above stated and the proof of Lemma~\ref{lem:4-moment} postponed until Section~\ref{sec:4-moment}, Lemma~\ref{lem:deviation} is easily deduced as follows. 

\begin{proofof}{Lemma~\ref{lem:deviation}}
\begin{align*}
\Pr \Big\{\inner{\bx}{\bz}^2 \geq \theta \Ex \left\{\inner{\bx}{\bz}^2 \right\} \Big\} \geq (1-\theta)^2 \frac{\Ex \left\{ \inner{\bx}{\bz}^2\right\}^2}{\Ex \left\{ \inner{\bx}{\bz}^4\right\}} \geq (1-\theta)^2\frac{\mu_2(\bx)^2}{\tilde K(\bx)},
\end{align*}
where the first inequality is owing to the Paley-Zygmund inequality and the second relies on Lemma~\ref{lem:4-moment}. 
\end{proofof}

\subsubsection{Proof of Lemma~\ref{lem:4-moment}} \label{sec:4-moment} The following result facilitates our proof of Lemma~\ref{lem:4-moment}. 

\begin{lemma}\label{lem:4-moment-2}
Let $\bx \in \R^n$ be $2$-central and $4$-bounded
and let $\bz \in \R^n$. Then, 
\begin{align*}
\Ex &\left\{\inner{\bx}{\bz}^4 \right\}  \leq \\ & 3 \mu_2(\bx)^2 \|\bz\|_2^4 + \left(\mu_4(\bx) - 3 \mu_2(\bx)^2 \right)\|\bz\|_4^4 + 4 \mu_3(\bx) \inner{\bmu}{\bz} \sum_{i=1}^n |\bz_i^3| + 6 \mu_2(\bx) \inner{\bmu}{\bz}^2 \|\bz\|_2^2 + \inner{\bmu}{\bz}^4,
\end{align*}
where $\bmu := \Ex\{\bx\}$.
\end{lemma}

\begin{proof}
Set $\br := \bx - \bmu$ and note that $\br$ is centred. Hence,
\begin{align}
\Ex & \left\{\inner{\bx}{\bz}^4 \right\} =
\Ex \left\{\left(\inner{\br}{\bz} + \inner{\bmu}{\bz} \right)^4 \right\}= \nonumber \\ 
& \Ex\left\{\inner{\br}{\bz}^4 \right\} + 4 \Ex \left\{\inner{\br}{\bz}^3 \inner{\bmu}{\bz} \right\} + 6 \Ex \left\{\inner{\br}{\bz}^2\inner{\bmu}{\bz}^2 \right\} + 4 \Ex \left\{\inner{\br}{\bz}\inner{\bmu}{\bz}^3 \right\} + \inner{\bmu}{\bz}^4. \label{eq:terms}
\end{align}

We proceed to bound each term appearing on the right hand side of~\eqref{eq:terms} separately. For the first term, write
\begin{align*}
\Ex\left\{\inner{\br}{\bz}^4 \right\} & = \sum_{i,j,k,\ell \in [n]} \Ex\{\br_i\br_j\br_k\br_\ell \}\bz_i\bz_j\bz_k\bz_\ell \\
& = \sum_{i \in [n]} \Ex\left\{\br_i^4 \right\}\bz_i^4 + 3\sum_{\substack{i,j \in [n]\\ i\neq j}}\Ex\left\{\br_i^2\right\}\Ex\left\{\br_j^2\right\} \bz_i^2 \bz_j^2 \\
& \leq \left(\mu_4(\bx) - 3\mu_2(\bx)^2 \right)\|\bz\|_4^4 +3 \mu_2(\bx)^2\|\bz\|_2^4. 
\end{align*}
For the other terms, the following equalities 
\begin{align*}
4 \Ex \left\{\inner{\br}{\bz}^3 \inner{\bmu}{\bz} \right\} & = 4 \inner{\bmu}{\bz} \sum_{i,j,k \in [n]} \Ex\left\{\br_i\br_j\br_k \right\} \bz_i \bz_j\bz_k  \leq 4 \mu_3(\bx) \inner{\bmu}{\bz} \sum_{i \in [n]} |\bz_i^3|,\\
6 \Ex \left\{\inner{\br}{\bz}^2\inner{\bmu}{\bz}^2 \right\} & = 6 \mu_2(\bx)\inner{\bmu}{\bz}^2 \|\bz\|_2^2,\\ 
4 \Ex \left\{\inner{\br}{\bz}\inner{\bmu}{\bz}^3 \right\} & = 4 \inner{\bmu}{\bz}^3 \sum_{i \in [n]} \Ex\{\br_i\} \bz_i = 0 
\end{align*}
are observed. The claim follows. 
\end{proof}

We are now in position to prove Lemma~\ref{lem:4-moment}. 

\begin{proofof}{Lemma~\ref{lem:4-moment}} 
Set $\bmu := \Ex \{\bx\}$. Observation~\ref{lem:inner-prod}, coupled with the assumption that $\|\bz\|_2 =1$, renders 
\begin{equation}\label{eq:2nd-moment-1}
\Ex \left\{ \inner{\bx}{\bz}^2 \right\}^2 = \left(\mu_2(\bx) + \inner{\bmu}{\bz}^2\right)^2 = \mu_2(\bx)^2 + 2 \mu_2(\bx)\inner{\bmu}{\bz}^2 + \inner{\bmu}{\bz}^4.
\end{equation} 
Since $\|\bz\|_2 = 1$ and $\|\bz\|_4^4 \leq \|\bz\|_2^4$, it follows by Lemma~\ref{lem:4-moment-2} that  
\begin{equation}\label{eq:4th-moment}
\Ex\left\{\inner{\bx}{\bz}^4 \right\} \leq  \mu_4(\bx) + 4 \mu_3(\bx) \inner{\bmu}{\bz} \sum_{i=1}^n |\bz_i^3| + 6 \mu_2(\bx) \inner{\bmu}{\bz}^2 + \inner{\bmu}{\bz}^4.
\end{equation}
Equality~\eqref{eq:2nd-moment-1} affords us with
\begin{equation}\label{eq:lower-bound}
\Ex \left\{ \inner{\bx}{\bz}^2 \right\}^2 \geq 2 \mu_2(\bx)\inner{\bmu}{\bz}^2 + \inner{\bmu}{\bz}^4
\end{equation}
which in turn allows us to rewrite~\eqref{eq:4th-moment} as follows
$$
\Ex\left\{\inner{\bx}{\bz}^4 \right\}  \leq \Ex \left\{ \inner{\bx}{\bz}^2 \right\}^2 + \mu_4(\bx) + 4 \mu_3(\bx)\inner{\bmu}{\bz}\sum_{i=1}^n |\bz_i^3| + 4\mu_2(\bx)\inner{\bmu}{\bz}^2.
$$
The inequality $\Ex \left\{ \inner{\bx}{\bz}^2 \right\}^2 \geq 2 \mu_2(\bx)\inner{\bmu}{\bz}^2$, supported by~\eqref{eq:lower-bound}, allows for
$$
\Ex\left\{\inner{\bx}{\bz}^4 \right\} \leq 3\Ex \left\{ \inner{\bx}{\bz}^2 \right\}^2 + \mu_4(\bx) + 4 \mu_3(\bx)\inner{\bmu}{\bz}\sum_{i=1}^n |\bz_i^3|.
$$
Noticing that $\|\bz\|_2 =1$ implies $\sum_{i=1}^n |\bz_i^3| \leq \|\bz\|_2^2 =1$, and that $\inner{\bmu}{\bz} \leq 1 + \inner{\bmu}{\bz}^2$, the latter supported by the fact that $x \leq 1 + x^2$ holds for all $x \in \R$, we may proceed to  write 
\begin{align*}
\Ex\left\{\inner{\bx}{\bz}^4 \right\} & \leq 3\Ex \left\{ \inner{\bx}{\bz}^2 \right\}^2 + \mu_4(\bx) + 4 \mu_3(\bx) + 4 \mu_3(\bx)\inner{\bmu}{\bz}^2 \\
& = 3\Ex \left\{ \inner{\bx}{\bz}^2 \right\}^2 +\frac{\mu_4(\bx)\mu_2(\bx)^2 + 4 \mu_3(\bx)\mu_2(\bx)^2 + 4 \mu_3(\bx)\mu_2(\bx)^2\inner{\bmu}{\bz}^2}{\mu_2(\bx)^2} \\
& \leq \frac{3\mu_2(\bx)^2 + 4 \mu_3(\bx) + 2 \mu_3(\bx) \mu_2(\bx) + \mu_4(\bx)}{\mu_2(\bx)^2} \cdot \Ex \left\{ \inner{\bx}{\bz}^2 \right\}^2 \\
& = \frac{\tilde K(\bx)}{\mu_2(\bx)^2} \cdot \Ex \left\{ \inner{\bx}{\bz}^2 \right\}^2
\end{align*}
as required, where in the last inequality we use the fact that $\Ex \left\{ \inner{\bx}{\bz}^2 \right\}^2 \geq \max \left\{\mu_2(\bx)^2, 2 \mu_2(\bx) \langle \bmu, \bz \rangle^2 \right \}$, which is supported by~\eqref{eq:2nd-moment-1}.
\end{proofof}

\section{Empirical width: Proof of Proposition~\ref{prop::W}}\label{sec:Rademacher} In this section, we prove Proposition~\ref{prop::W}. 
Start by recalling that $\Sigma_s = \{\bu \in \R^n: \|\bu\|_0 \leq s\}$ and set $\Sigma_s^2 := \{\bu \in \Sigma_s: \|\bu\|_2 \leq 1\}$. Proceed to write 
\begin{align}
W(T_{\rho,s}; \ba_1,\ldots,\ba_m) & \overset{\phantom{\eqref{eq:h}}}{=} \Ex \left\{ \sup_{\bu \in T_{\rho,s}} |\inner{\bh}{\bu}| \right\}\nonumber\\ 
& \overset{\phantom{\eqref{eq:h}}}{\leq} \frac{3}{\rho} \Ex \left\{ \sup_{\bu \in \Sigma_s^2} |\inner{\bh}{\bu}|\right\} \nonumber\\ 
& \overset{\eqref{eq:h}}{=} \frac{3}{\rho} \Ex \left\{ \sup_{\bu \in \Sigma_s^2} \left| \frac{1}{\sqrt{m}} \sum_{i=1}^m \xi_i \inner{\ba_i}{\bu} \right|\right\} \nonumber\\
& \overset{\phantom{\eqref{eq:h}}}{=} \frac{3}{\rho} \Ex \left\{ \sup_{\bu \in \Sigma_s^2}  \frac{1}{\sqrt{m}} \sum_{i=1}^m \xi_i \inner{\ba_i}{\bu} \right\},\label{eq:common}
\end{align}
where the above inequality is supported, first, by \cite[Lemma~3.2]{SGR15} asserting that $\mathrm{conv}(T_{\rho,s}) \subseteq \left(2+\rho^{-1}\right) \mathrm{conv} (\Sigma_s^2)$ (where we use $\mathrm{conv(X)}$ to denote the convex hull of a set $X$), and, second, by the fact that 
$$
\sup_{\bx \in \mathrm{conv}(C)} |\inner{\bx}{\bz}| = \sup_{\bx \in C} |\inner{\bx}{\bz}|
$$
holds, whenever $C \subseteq \R^n$ and $\bz \in \R^n$; the latter is often referred to as {\sl Bauer's Maximum Principle}~\cite{Bauer}.  For the last equality, we rely on the symmetry of $\Sigma_s^2$, by which we mean that $\bv \in \Sigma_s^{2}$ mandates that $-\bv \in \Sigma_s^2$ holds as well.

Continuing from~\eqref{eq:common}, for every $i \in [m]$ set $\bx_i := \ba_i - \bmu_i$, where $\bmu_i := \Ex \{\ba_i\}$. Then  
\begin{align}
W(T_{\rho,s}; \ba_1,\ldots,\ba_m)  & \leq \frac{3}{\rho} \Ex \left\{\sup_{\bu \in \Sigma_s^2} \frac{1}{\sqrt{m}} \sum_{i=1}^m \xi_i\inner{\bx_i+\bmu_i}{\bu} \right\} \nonumber \\
& \leq \frac{3}{\rho} \Ex \left\{\sup_{\bu \in \Sigma_s^2} \frac{1}{\sqrt{m}} \sum_{i=1}^m \xi_i \inner{\bx_i}{\bu}  \right\}+ \frac{3}{\rho} \Ex \left\{\sup_{\bu \in \Sigma_s^2} \frac{1}{\sqrt{m}} \sum_{i=1}^m \xi_i \inner{\bmu_i}{\bu}  \right\}. \label{eq:W-decomposition}
\end{align}

Proposition~\ref{prop::W} is then implied by the following lemmata providing estimates for the two terms appearing on the right hand side of~\eqref{eq:W-decomposition}.    

\begin{lemma}\label{lem::X-part}
Let $s \in [n]$. Then, 
$$
\Ex \left\{ \sup_{\bu \in \Sigma_s^2} \frac{1}{\sqrt{m}} \sum_{i=1}^m \xi_i \inner{\bx_i}{\bu} \right\} = O\left(e^{2\alpha} \kappa \sqrt{s \log(en/s)} \right)  
$$
holds, whenever $m \geq \left(\log (en) \right)^{\max\{2\alpha-1,1\}}$. 
\end{lemma}

\begin{lemma}\label{lem::Z-part}
 $$
 \Ex \left\{\sup_{\bu \in \Sigma_s^2} \frac{1}{\sqrt{m}} \sum_{i=1}^m \xi_i \inner{\bmu_i}{\bu}  \right\} 
 \leq F(A).
 $$   
\end{lemma}

\noindent
Lemmas~\ref{lem::X-part} and~\ref{lem::Z-part} are proved in Sections~\ref{sec:X-part} and~\ref{sec:Z-part}, respectively. \hfill\QED

\subsection{Proof of Lemma~\ref{lem::X-part}}\label{sec:X-part}
For a vector $\by \in \R^n$, write $\by^*$ to denote the vector in $\R^n$ obtained through a monotone non-increasing rearrangement of the sequence $(|\by_i|)_{i \in [n]}$. Given $s \in [n]$, set 
$$
\|\by^*\|^2_{2,s}:= \sum_{i=1}^s (\by^*_i)^2. 
$$
If, in addition, $\by$ is random, then 
\begin{equation}\label{eq:inner-re}
\Ex\left\{ \sup_{\bu \in \Sigma_s^2} \inner{\by}{\bu} \right\} = \Ex\{\|\by^*\|_{2,s}\},
\end{equation}
holds whenever $s \in [n]$. To see~\eqref{eq:inner-re}, fix $S \subseteq [n]$ satisfying $|S| = s$ and note that 
$$
\sup_{\substack{\bu \in \Sigma_s^2 \\ \supp (\bu) = S}} \inner{\by}{\bu}= \inner{\by_S}{\frac{\by_S}{\|\by_S\|_2}} = \|\by_S\|_2.
$$
This in turn allows for 
$$
\sup_{\bu \in \Sigma_s^2} \inner{\by}{\bu} = \sup_{S \subseteq [n], |S| =s} \sup_{\substack{\bu \in \Sigma_s^2 \\ \supp(\bu) = S} } \inner{\by}{\bu} = \sup_{S \subseteq [n], |S|=s} \|\by_S\|_2 = \|\by^*\|_{2,s}. 
$$
Taking expectations on both sides of the last equality delivers~\eqref{eq:inner-re}. 

\medskip
The following result is an adaptation of~\cite[Lemma~6.5]{M15} stripping from the latter the requirement for distributional identity to be held by the vector entries. 

\begin{lemma}\label{lem:2-norm-re}{\em~\cite[Lemma~6.5]{M15}}
Let $\kappa >0$ and let $\by \in \R^n$ be a random vector with independent entries, satisfying the property that $\|\by_i\|_p \leq \kappa \sqrt{p}$ holds whenever $2 \leq p \leq \log(en)$ and $i \in [n]$. Then,
$$
\Ex\{\|\by^*\|_{2,s}\} = O \left(\kappa \sqrt{s \log (en/s)} \right)
$$
holds, whenever $s\in [n]$. 
\end{lemma}

\noindent
The proof of Lemma~\ref{lem:2-norm-re} is postponed until Appendix~\ref{app:2-norm-re}. The last ingredient facilitating our proof of Lemma~\ref{lem::X-part} reads as follows.  

\begin{lemma}\label{lem:sum-p-norm}{\em~\cite[Lemma~2.8]{LM17}}
Let $\by \in \R^\ell$ be a zero-mean $(\kappa,\alpha,r)$-reasonable vector with independent entries for some $\kappa > 0$, $\alpha \geq 1/2$, and $r \geq 2$. Then, there exists an absolute constant $C>0$ such that 
$$
\left\|\frac{1}{\sqrt{\ell}} \sum_{i=1}^\ell \by_i \right\|_p \leq C e^{2\alpha} \kappa \sqrt{p}
$$
holds, whenever $\ell \geq r^{\max\{2\alpha-1,1\}}$ and $2 \leq p \leq r$. 
\end{lemma}

\begin{remark}
Lemma~\ref{lem:sum-p-norm} is an adaptation of~\cite[Lemma~2.8]{LM17} originally formulated for 
random vectors with centred i.i.d. entries. To be rid of the latter assumption, start with proving that 
\begin{equation}\label{eq:Latala}
\left\|\sum_{i=1}^n X_i\right\|_p = O \left(\max\left\{\frac{p}{s}\left(\frac{n}{p} \right)^{1/s}\max_{i \in [n]} \|X_i\|_s: 2 \leq s \leq p \right\} \right)    
\end{equation}
holds, whenever $X_1,\ldots,X_n$ are independent centred random variables and $2 \leq p \leq r$. A standard application of symmetrisation\footnote{$\frac{1}{2}\left\|\sum_i X_i\right\|_p \leq \left\|\sum_i \xi_iX_i\right\|_p \leq 2\left\|\sum_i X_i\right\|_p$, with $(\xi_i)_{i \in [n]}$ a sequence of independent Rademacher random variables that are also independent of the sequence $(X_i)_{i \in [n]}$.}, allows one to further assume that the members of $(X_i)_{i\in[n]}$, seen in~\eqref{eq:Latala}, are all symmetric. The formulation of~\eqref{eq:Latala} thus obtained coincides with that seen for this equality in~\cite[Corollary~2]{Lat97} modulo an i.i.d. assumption present in the latter. As asserted in~\cite{Lat97}, the symmetric case of~\eqref{eq:Latala} readily follows from a proof of the latter under the assumption that the members of $(X_i)_{i \in[n]}$ are non-negative. A proof of the non-negative case of~\eqref{eq:Latala} that is devoid of an i.i.d. assumption is that seen for~\cite[Corollary~1]{Lat97} essentially verbatim and thus omitted. Equipped with~\eqref{eq:Latala}, the argument of~\cite[Lemma~2.8]{LM17} is then easily adapted as to remove its i.i.d. assumption leading to the formulation seen in Lemma~\ref{lem:sum-p-norm}.  
\end{remark}

\begin{proofof}{Lemma~\ref{lem::X-part}}
Define the random vector $\bv \in \R^n$, given by 
$$
\bv_j : = \frac{1}{\sqrt{m}} \sum_{i=1}^m \xi_i (\bx_i)_j.
$$
Conditioning over the Rademacher random variables $\bxi := (\xi_i)_{i\in [m]}$, the entries of $\bv$ become independent of one another. Note that
$$
\Ex \left\{\sup_{\bu \in \Sigma_s^2} \frac{1}{\sqrt{m}} \sum_{i=1}^m \xi_i \inner{\bx_i}{\bu} \mid \bxi \right\} = \Ex\left\{\sup_{\bu \in \Sigma_s^2} \inner{\bv}{\bu} \mid \bxi \right\} \overset{\eqref{eq:inner-re}}{=} \Ex\left\{\|\bv^*\|_{s,2} \mid \bxi\right\}.
$$

By assumption, $m \geq (\log (en))^{\max\{2\alpha-1,1\}}$ holds for any $s \in [n]$. Additionally, it is  assumed that $\|\xi_i (\bx_i)_j\|_p = \|(\bx_i)_j\|_p \leq \kappa p^\alpha$ holds, whenever $1 \leq p \leq \log (en)$. Noticing that $\bv$ conditioned on $\bxi$ forms a zero-mean vector with independent entries, Lemma~\ref{lem:sum-p-norm} (applied with $\ell = m$ and $r = \log (en)$) then asserts the existence of a constant $C>0$ for which 
$$
\|\bv_j\mid \bxi\|_p \leq Ce^{2\alpha}\kappa \sqrt{p}
$$
holds, whenever $2 \leq p \leq \log (en)$. It follows that the vector $\bv$, conditioned on a realisation of $\bxi$, satisfies the premise of Lemma~\ref{lem:2-norm-re}; the latter then asserts that 
$$
\Ex\left\{\|\bv^*\|_{s,2} \mid \bxi\right\} = O\left(e^{2 \alpha} \kappa \sqrt{s \log (en/s)} \right).
$$
The Law of Total Expectation then concludes the proof through
\begin{align*}
\Ex \left\{\sup_{\bu \in \Sigma_s^2} \frac{1}{\sqrt{m}} \sum_{i=1}^m \xi_i \inner{\bx_i}{\bu} \right\} & = \sum_{\br \in \{\pm 1\}^m} \Ex\left\{\sup_{\bu \in \Sigma_s^2} \frac{1}{\sqrt{m}} \sum_{i=1}^m \xi_i \inner{\bx_i}{\bu} \mid \bxi = \br\right\}  \Pr\left\{\bxi = \br \right\}\\
& = \sum_{\br \in \{\pm 1\}^m} \Ex \left\{\|\bv^*\|_{s,2} \mid \bxi = \br \right\}\Pr\left\{\bxi= \br \right\} \\
& = O\left(e^{2 \alpha} \kappa \sqrt{s \log (en/s)} \right).
\end{align*}
\end{proofof}

\subsection{Proof of Lemma~\ref{lem::Z-part}}\label{sec:Z-part}
Start by proving that
\begin{align*}
    \Ex\left\{\sup_{\bu \in \Sigma_s^2} \frac{1}{\sqrt{m}} \sum_{i=1}^m \xi_i \inner{\bmu_i}{\bu}\right\} & \overset{\phantom{\eqref{eq:inner-re}}}{\leq} 
    \sqrt{s\log n} \|\Ex\{A\}\|_{\infty}.
\end{align*}
Note that
\begin{align}
\Ex \left\{\sup_{\bu \in \Sigma_s^2} \frac{1}{\sqrt{m}} \sum_{i=1}^m \xi_i \inner{\bmu_i}{\bu} \right\}  & = \Ex \left\{\sup_{\bu \in \Sigma_s^2} \inner{\bu}{\frac{1}{\sqrt{m}} \sum_{i=1}^m \xi_i \bmu_i} \right\}\nonumber\\
& \leq \Ex \left\{\sup_{\bu \in \Sigma_s^2} \left|\inner{\bu}{\frac{1}{\sqrt{m}} \sum_{i=1}^m \xi_i \bmu_i}\right| \right\}\nonumber\\
& \leq \Ex \left\{\sup_{\bu \in \Sigma_s^2} \|\bu\|_1 \left\| \frac{1}{\sqrt{m}} \sum_{i=1}^m \xi_i \bmu_i\right\|_\infty \right\}\nonumber \\
& \leq \sqrt{s} \cdot \Ex \left\{ \max_{j \in [n]} \left|\frac{1}{\sqrt{m}} \sum_{i=1}^m \xi_i(\bmu_i)_j \right| \right\} \label{eq:sub-exp-Dudley-setup},
\end{align}
where the penultimate inequality is owing to H\"older's inequality and where the last inequality relies on $\|\bu\|_1 \leq \sqrt{s} \|\bu\|_2 = \sqrt{s}$ supported by the Cauchy-Schwartz inequality.  

Gearing up towards an application of~\eqref{eq:max-sub-gauss}, set the random variables 
$$
Z_j := \frac{1}{\sqrt{m}}\sum_{i=1}^m \xi_i (\bmu_i)_j,
$$
for every $j \in [n]$. Note that $Z_j$ is a sum of (scaled) independent Rademacher random variables and therefore is sub-Gaussian with the following sub-Gaussian norm 
\begin{equation}\label{eq:sub-exp-norm-Z}
\|Z_j\|_{\psi_2}^2  \overset{\eqref{eq:sub-gauss-sum}}{=} \frac{1}{m} \cdot O\left(\sum_{i=1}^m \|\xi_i (\bmu_i)_j\|_{\psi_2}^2 \right) 
= \frac{1}{m} \sum_{i=1}^m O\left( (\bmu_i)_j^2\|\xi_i\|^2_{\psi_2}\right) = O\left(\max_{i\in [m]} \|\Ex\{A\}\|_{\infty}^2 \right) 
\end{equation}
holds, where the last inequality holds since $\| \xi_i \|_{\psi_2}^2 \leq 1$ holds for every $i \in [m]$. 

Therefore, 
\begin{align*}
\Ex \left\{\sup_{\bu \in \Sigma_s^2} \frac{1}{\sqrt{m}} \sum_{i=1}^m \xi_i \inner{\bmu_i}{\bu} \right\} & \overset{\eqref{eq:sub-exp-Dudley-setup}}{\leq} \sqrt{s} \cdot \Ex \left\{ \max_{j \in [n]} \left|\frac{1}{\sqrt{m}} \sum_{i=1}^m \xi_i(\bmu_i)_j \right| \right\} \\
& \overset{\phantom{\eqref{eq:sub-exp-Dudley-setup}}}{=} \sqrt{s} \cdot \Ex \left\{ \max_{j \in [n]} |Z_j| \right\} \\
& \overset{\phantom{\eqref{eq:sub-exp-Dudley-setup}}}{=} \sqrt{s} \cdot 
O\left(\sqrt{\log n} \cdot \max_{\substack{i\in [m]\\j \in [n]}} \|\Ex\{A\}\|_{\infty} \right), 
\end{align*}
where for the last equality we rely on~\eqref{eq:max-sub-gauss} and~\eqref{eq:sub-exp-norm-Z}.

\bigskip
Moving on to the second bound in the definition of $F(A)$, start by writing
\begin{align*}
\Ex\left\{\sup_{\bu \in \Sigma_s^2} \frac{1}{\sqrt{m}} \sum_{i=1}^m \xi_i \inner{\bmu_i}{\bu}\right\} & \overset{\phantom{\eqref{eq:inner-re}}}{=} \Ex \left\{\sup_{\bu \in \Sigma_s^2} \frac{1}{\sqrt{m}} \sum_{i=1}^m \xi_i \sum_{j=1}^n (\bmu_i)_j \bu_j \right\}\\
& \overset{\phantom{\eqref{eq:inner-re}}}{=} \Ex \left\{\sup_{\bu \in \Sigma_s^2} \sum_{j=1}^n \left(\frac{1}{\sqrt{m}}\sum_{i=1}^m \xi_i(\bmu_i)_j \right)\bu_j \right\}\\
& \overset{\phantom{\eqref{eq:inner-re}}}{=} \Ex \left\{\sup_{\bu \in \Sigma_s^2} \inner{\bz}{\bu} \right\}\\
& \overset{\eqref{eq:inner-re}}{=} \Ex \left\{\|\bz^*\|_{2,s}\right\},
\end{align*}
where $\bz \in \R^n$ is given by 
$$
\bz_j := \frac{1}{\sqrt{m}}\sum_{i=1}^m \xi_i(\bmu_i)_j.
$$

As above, set $\bxi := (\xi_i)_{i\in [m]}$. For $j \in [n]$, let $\bpsi_j$ denote the $j$th column of $\Ex\{A\}$ and proceed to write 
\begin{align*}
\Ex \left\{\|\bz^*\|_{2,s}\right\} & = \Ex\left\{ \max_{\substack{S \subseteq [n]\\ |S|=s}}\sqrt{\sum_{j \in S} \left( \frac{1}{\sqrt{m}} \sum_{i=1}^m \xi_i (\bmu_i)_j \right)^2} \right\} \\
& = \Ex \left\{ \max_{\substack{S \subseteq [n]\\ |S|=s}} \sqrt{\frac{1}{m} \sum_{j\in S} \inner{\bxi}{\bpsi_j}^2}\right\} \\
& \leq \Ex \left\{ \max_{\substack{S \subseteq [n]\\ |S|=s}} \sqrt{\frac{1}{m} \sum_{j\in S} \|\bxi\|_2^2 \|\bpsi_j\|_2^2}\right\} \\
& = \Ex \left\{ \max_{\substack{S \subseteq [n]\\ |S|=s}} \sqrt{\sum_{j\in S} \|\bpsi_j\|_2^2} \right\} \\
& = \|\Ex\{A\}\|_{\frob,s},
\end{align*}
where the above inequality holds by the Cauchy-Schwartz inequality and the last equality follows by the definition of $\|\cdot\|_{\frob,s}$ and the fact that the quantity seen in the penultimate equality involves no randomness. 
\hfill\QED

\section{Concluding remarks and further research}

Our main result, namely Theorem~\ref{thm::main-perturb}, together its various corollaries, namely Corollaries~\ref{cor:classical}, ~\ref{cor:classical-perturb}, and~\ref{cor::main-perturb} collectively assert that fairly arbitrary matrices $M \in \R^{m \times n}$ can a.a.s. be mended, so to speak, into matrices satisfying $\ell_2$-RNSP once these are additively randomly perturbed entry-wise using a random matrix $R \in \R^{m\times n}$. A wide range of distributions for the entries of $R$ is considered, allowing the latter to be sub-gaussian, sub-exponential, or merely  reasonable (i.e. rather heavy-tailed). This in turn captures and generalises previous results seen in~\cite{SGR15,KJ16b,KJ16,LM17,SJC19,SJC19b}.

A natural next step following our results here is to consider NSPs in the multiplicative random perturbation model. More concretely, given a fixed matrix $M \in \R^{m \times n}$ as well as a random matrix $R \in \R^{n \times d}$, determine whether $MR$ a.a.s. satisfies $\ell_2$-RNSP; as in our work here, the nature of the distributions of the entries of $R$ is of prime interest. 

New subtleties and challenges arise in the multiplicative perturbation model. The kernels of $NM$ and $M$ coincide, whenever $N \in R^{m \times m}$ is non-singular; this in turn implies that NSP is preserved under left multiplication by a non-singular matrix. As~\cite[Exercise~4.2]{FR13} demonstrates, matrices $MD$, with $D$ a conformal non-singular diagonal matrix, need not satisfy NSP despite $M$ satisfying this property. A statistical study of the NSPs of random matrices $MR$, as above, probes into the ubiquity (or lack thereof) of the latter phenomenon. 

The multiplicative perturbation model introduces dependencies between the entries of $MR$ to such an extent that the applicability of Mendelson's method comes into question as the latter can no longer be applied to $MR$ directly as in the additive perturbation model. 

\bibliographystyle{amsplain} 
\bibliography{Lit}
 
\appendix

\section{Properties of distributions}\label{app:dist}

In this section, we collect various properties of sub-gaussian and sub-exponential distributions utilised throughout our arguments above; our account here, follows that of~\cite[Chapter~2]{Vershynin}. The {\em sub-gaussian norm} of a random variable $X$ is given by 
$$
\|X\|_{\psi_2} := \inf \left\{t >0: \Ex \left\{\exp\left(\frac{X^2}{t^2} \right)\right\} \leq 2\right\} 
$$
and its {\em sub-exponential} norm is defined to be
$$
\|X\|_{\psi_1} := \inf \left\{t >0: \Ex \left\{\exp\left(\frac{|X|}{t} \right)\right\} \leq 2\right\}. 
$$

For a bounded random variable $X$,  
\begin{equation}\label{eq:sub-gauss-bounded}
\|X\|_{\psi_1},\|X\|_{\psi_2} = O(\|X\|_\infty)
\end{equation}
holds. By~\cite[Proposition~2.6.1]{Vershynin}, sub-gaussianity is inherited by sums of zero-mean independent random variables in the sense of 
\begin{equation}\label{eq:sub-gauss-sum}
\left\|\sum_{i=1}^n  X_i\right\|_{\psi_2}^2 = O\left(\sum_{i=1}^n \|X_i\|_{\psi_2}^2\right).
\end{equation}

Lastly, given a sequence $X_1,\ldots,X_n$ of sub-gaussian random variables, a well-known result, see e.g.~\cite[Exercise~2.5.10]{Vershynin}, asserts that 
\begin{equation}\label{eq:max-sub-gauss}
\Ex\left\{\max_{i \in [n]} |X_i| \right\} = O\left(\sqrt{\log n} \cdot \max_{i\in[n]}\left\{ \|X_i\|_{\psi_2}\right\} \right).
\end{equation}

\section{Adaptation of Mendelson's method:\\ Proof of Theorem~\ref{thm::Mendelson}}\label{app:Mendelson}

%

Our proof of Theorem~\ref{thm::Mendelson} follows closely the account of Tropp~\cite{T15} for Mendelson's small-ball method~\cite[Theorem~5.4]{M15}. The following two results facilitate our proof of Theorem~\ref{thm::Mendelson}. The first is a standard consequence of symmetrisation that can be derived from~\cite[Lemma~6.3]{LT91}; see e.g. the proof of~\cite[Lemma~6.6]{LT91} as well as~\cite[Exercise~6.4.4]{Vershynin}.

\begin{lemma}\label{lem:sym}
Let $ X_1, X_2, \ldots, X_n $ be independent random variables  and let \( \mathcal{F} \subseteq \R^n \) be a set of functions (viewed as vectors). Then,  
\[
\Ex\left\{ \sup_{f \in \mathcal{F}} \left|  \sum_{i=1}^n \Big( f(X_i) - \Ex\left\{f(X_i)\right\} \Big) \right|\right\}
\leq 2 \, \Ex \left\{ \sup_{f \in \mathcal{F}} \left| \sum_{i=1}^n \xi_i f(X_i) \right|\right\},
\]
where $ (\xi_i)_{i \in [n]}$ are independent Rademacher random variables which are also independent of the sequence $\left( X_i\right)_{i \in [n]}$.
\end{lemma}

The second result facilitating our proof of Theorem~\ref{thm::Mendelson} is the so-called {\sl bounded differences} concentration inequality often dubbed as {\sl McDiarmid's inequality}, who was the first~\cite{M89} to phrase it in an effective and easy to apply form. A multivariate function $f: \mathcal{X}^n \to \R$ is said to possess the {\em bounded differences property} with non-negative parameters $c_1,\ldots,c_n$ provided 
$$
\sup_{z, y_1,\ldots,y_n \in \mathcal{X}} \Big|f(y_1,\ldots,y_n) - f(y_1,\ldots,y_{i-1},z,y_{i+1},\ldots,y_n) \Big| \leq c_i, 
$$
holds, whenever $ i \in [n]$.  

\begin{theorem}{~\cite[Theorem~6.2]{BLM13}}
\label{thm:BoundedDiff}
Let $f:\mathcal{X}^n \rightarrow \R$ be a function satisfying the bounded differences property with non-negative parameters $c_1,\ldots,c_n$ and let
    $$
        v := \frac{1}{4}\sum_{i=1}^n c_i^2.
    $$
    Let $X_1,\ldots, X_n$ be a sequence of independent random variables and set $Z:=f(X_1,\ldots,X_n)$. Then,
    for every $t>0$,
    $$
        \Pr\left\{Z \geq \Ex{Z} + t\right\}\leq e^{-t^2/2v} 
    $$
holds.
\end{theorem}

\begin{proofof}{Theorem~\ref{thm::Mendelson}}
Given $\eps >0$, observe that 
\begin{equation}\label{eq:lower-bound-any-vector}
\left(\sum_{i=1}^m \inner{\bphi_i}{\bu}^2 \right)^{1/2} \geq \frac{1}{\sqrt{m}} \sum_{i=1}^m \left|\inner{\bphi_i}{\bu} \right| \geq \frac{\eps}{\sqrt{m}} \sum_{i=1}^m \ind\left\{\left|\inner{\bphi_i}{\bu} \right| \geq \eps\right\}
\end{equation}
holds for any $\bu \in \R^n$, where the first inequality relies on the Cauchy-Schwartz inequality and the second inequality relies on the fact that 
$
\left|\inner{\bphi_i}{\bu} \right|\geq \eps \cdot \ind\left\{\left|\inner{\bphi_i}{\bu} \right| \geq \eps \right\}
$ holds. 
Set $\tilde Q_{2\eps} (\bphi_i,\bu) := \Pr \left\{\left|\inner{\bphi_i}{\bu} \right| \geq 2\eps \right\} $ and write 
\begin{align}
\inf_{\bu \in E} &\left(\sum_{i=1}^m \inner{\bphi_i}{\bu}^2 \right)^{1/2} \nonumber \\
& \overset{\eqref{eq:lower-bound-any-vector}}{\geq} \inf_{\bu \in E} \left(\frac{\eps}{\sqrt{m}} \sum_{i=1}^m \tilde Q_{2\eps}(\bphi_i,\bu) - \frac{\eps}{\sqrt{m}} \sum_{i=1}^m \Big(\tilde Q_{2\eps}(\bphi_i,\bu) - \ind\left\{\left|\inner{\bphi_i}{\bu} \right|\geq \eps \right\} \Big) \right)\nonumber \\
& \overset{\eqref{eq:Q}}{\geq} \eps \sqrt{m} Q_{2\eps}(E; \bphi_1,\ldots,\bphi_m) - \frac{\eps}{\sqrt{m}} \sup_{\bu \in E} \sum_{i=1}^m \Big(\tilde Q_{2\eps}(\bphi_i,\bu) - \ind\left\{ \left|\inner{\bphi_i}{\bu} \right| \geq \eps \right\} \Big). \label{eq:initial-lower-bound}
\end{align}

Focusing on the second term appearing on the right hand side of~\eqref{eq:initial-lower-bound}, given $i \in [m]$ and $\bu \in E$, define the random variable 
$$
X_{i,\bu} := \tilde Q_{2\eps}(\bphi_i,\bu) - \ind\left\{\left|\inner{\bphi_i}{\bu} \right|\geq \eps\right\}
$$
and note that $X_{i,\bu} \in \left\{\tilde Q_{2\eps}(\bphi_i,\bu)-1,\tilde Q_{2\eps}(\bphi_i,\bu)\right\} \subseteq [-1,1]$ always holds. The expression $\sup_{\bu \in E} \sum_{i=1}^m X_{i,\bu}$ can be reproduced through a function satisfying the bounded differences property with parameters $0 \leq c_1,\ldots, c_m \leq 2$. Hence, for any $t \geq 0$, Theorem~\ref{thm:BoundedDiff} yields  
$$
\Pr \left\{ \sup_{\bu \in E} \sum_{i=1}^m X_{i,\bu} \leq \Ex \left\{ \sup_{\bu \in E} \sum_{i=1}^m X_{i,\bu} \right\} + \sqrt{m} \cdot t\right\} \geq 1 -\exp\left(-t^2/2 \right);
$$ 
more explicitly this asserts that 
\begin{align}
\sup_{\bu \in E}\sum_{i =1}^m \Big(\tilde Q_{2\eps}(\bphi_i,\bu) -&  \ind \left\{\left|\inner{\bphi_i}{\bu}\right|\geq \eps \right\} \Big) \nonumber \\ 
& \leq \Ex \left\{\sup_{\bu \in E}\sum_{i =1}^m \Big(\tilde Q_{2\eps}(\bphi_i,\bu) - \ind \left\{\left|\inner{\bphi_i}{\bu}\right|\geq \eps \right\} \Big) \right\} + \sqrt{m}\cdot t \label{eq:post-bounded-diff}
\end{align}
holds with probability at least $1 -\exp\left(-t^2/2 \right)$ whenever $t \geq 0$. 

The function $\psi_\eps: \R \to [0,1]$ given by 
$$
\psi_\eps(s) := 
\begin{cases}
0, & |s| \leq \eps,\\
(|s|-\eps)/\eps, & \eps < |s| \leq 2\eps,\\
1, & |s| >2\eps
\end{cases}
$$ 
is introduced in order to upper bound the expectation seen on the right hand side of~\eqref{eq:post-bounded-diff}. Prior to doing so, we record two features that this function satisfies. The first is that $\psi_\eps(\cdot)$ acts as a {\em soft indicator} function, by which we mean that
\begin{equation}\label{eq:soft}
\ind\left\{|s| \geq 2 \eps \right\} \leq \psi_\eps(s) \leq \ind \left\{|s| \geq \eps \right\}. 
\end{equation}
The second trait supported by $\psi_\eps(\cdot)$ is that its scaling $\eps\psi_\eps(\cdot)$ acts as a {\em contraction} in the sense that 
$$
|\eps\psi_\eps(x) - \eps \psi_\eps(y)| \leq |x-y|
$$
holds for every $x$ and $y$; this in turn implies that 
\begin{equation}\label{eq:contract}
\eps \psi_\eps(x) \leq |x|
\end{equation}
always holds.

Proceeding to upper bound the expectation seen on the right hand side of~\eqref{eq:post-bounded-diff} we obtain  
\begin{align}
\Ex \Big\{ \sup_{\bu \in E}\sum_{i =1}^m \Big(\tilde Q_{2\eps}(\bphi_i,\bu) - &\ind \left\{\left|\inner{\bphi_i}{\bu}\right|\geq \eps \right\} \Big) \Big\}  \nonumber \\ 
& \overset{\phantom{\eqref{eq:soft}}}{=} \Ex \left\{\sup_{\bu \in E}\sum_{i =1}^m \Big(\Ex \left\{\ind\left\{\left|\inner{\bphi_i}{\bu} \right|\geq 2\eps\right\} \right\} - \ind \left\{\left|\inner{\bphi_i}{\bu}\right|\geq \eps \right\} \Big) \right\} \nonumber \\
& \overset{\eqref{eq:soft}}{\leq} \Ex \left\{\sup_{\bu \in E}\sum_{i =1}^m \Big(\Ex \left\{\psi_\eps\left(\inner{\bphi_i}{\bu} \right) \right\} - \psi_\eps\left(\inner{\bphi_i}{\bu} \right) \Big) \right\} \nonumber\\
& \overset{\phantom{\eqref{eq:soft}}}{\leq} \Ex \left\{\sup_{\bu \in E}\Big|\sum_{i =1}^m  \Ex \left\{\psi_\eps\left(\inner{\bphi_i}{\bu} \right) \right\} - \psi_\eps\left(\inner{\bphi_i}{\bu} \right) \Big| \right\} \nonumber \\
& \overset{\phantom{\eqref{eq:soft}}}{\leq} 2 \Ex \left\{\sup_{\bu \in E} \left|\sum_{i=1}^m \xi_i \psi_\eps \left(\inner{\bphi_i}{\bu} \right) \right| \right\} \nonumber\\
& \overset{\eqref{eq:contract}}{\leq} \frac{2}{\eps} \Ex \left\{\sup_{\bu \in E} \left|\sum_{i=1}^m \xi_i \inner{\bphi_i}{\bu} \right| \right\},\label{eq:exp-bound-final}
\end{align}
where the first equality is supported by 
$\tilde Q_{2\eps}(\bphi_i,\bu)= \Ex \left\{\ind\left\{\left|\inner{\bphi_i}{\bu} \right|\geq 2\eps\right\} \right\}$, and 
the penultimate inequality is owing to Lemma~\ref{lem:sym} which introduces the sequence $(\xi_i)_{i \in [m]}$ of independent Rademacher random variables. For the last inequality, beyond an appeal to~\eqref{eq:contract} we also rely on the fact that $\sup_{\bu \in E} \left|\sum_{i=1}^m \xi_i \inner{\bphi_i}{\bu} \right|$ and $\sup_{\bu \in E} \left|\sum_{i=1}^m \xi_i \left|\inner{\bphi_i}{\bu}\right| \right|$ have the same distribution and thus, in particular, the same expectation.

Prior to concluding the proof, recall that $\bh = \frac{1}{\sqrt{m}} \sum_{i=1}^m \xi_i \bphi_i$ so that 
$$
\frac{1}{\sqrt{m}} \sum_{i=1}^m \xi_i \inner{\bphi_i}{\bu} = \frac{1}{\sqrt{m}} \sum_{i=1}^m \xi_i \sum_{j=1}^n \bu_j(\bphi_i)_j = \sum_{j=1}^n\left(\sum_{i=1}^m \frac{1}{\sqrt{m}}\xi_i\bphi_i \right)_j\bu_j = \sum_{j=1}^n \bh_j \bu_j = \inner{\bh}{\bu}.
$$   
Hence, for any $t \geq 0$, with probability at least $1- \exp\left(-t^2/2\right)$ it holds that  
\begin{align*}
\inf_{\bu \in E} \Big(\sum_{i=1}^m \inner{\bphi_i}{\bu}^2 \Big)^{1/2} &\overset{\phantom{\eqref{eq:W}}}{\geq} \eps \sqrt{m} Q_{2\eps}(E;\bphi_1,\ldots,\bphi_m) - \frac{2}{\sqrt{m}} \Ex \left\{ \sup_{\bu \in E} \left|\sum_{i=1}^m \xi_i \inner{\bphi_i}{\bu} \right| \right\} - \eps t \\
& \overset{\phantom{\eqref{eq:W}}}{=} \eps \sqrt{m} Q_{2\eps}(E;\bphi_1,\ldots,\bphi_m) - 2 \Ex \left\{\sup_{\bu \in E} \left|\inner{\bh}{\bu} \right| \right\} - \eps t\\
& \overset{\eqref{eq:W}}{=} \eps \sqrt{m} Q_{2\eps}(E;\bphi_1,\ldots,\bphi_m) - 2 W(E;\bphi_1,\ldots,\bphi_m) - \eps t
\end{align*} 
where the first inequality is owing to~\eqref{eq:initial-lower-bound},~\eqref{eq:post-bounded-diff}, and~\eqref{eq:exp-bound-final}.
\end{proofof}  

\section{Proof of Lemma~\ref{lem:2-norm-re}}\label{app:2-norm-re}
We follow the account of~\cite[Lemma~6.5]{M15}. Given $s \in [n]$, it suffices to prove that
\begin{equation}\label{eq:2nd-moment}
\Ex \left\{(\by^*_i)^2 \right\} = O\left(\kappa^2 \log(en/i) \right)    
\end{equation}
holds, whenever $i \in [s]$. Indeed, given~\eqref{eq:2nd-moment}, we may proceed to write
$$
\Ex\left\{\|\by^*\|_{2,s} \right\} = \Ex \left\{\sqrt{\sum_{i=1}^s (\by^*_i)^2}\right\} \leq \sqrt{\Ex\left\{\sum_{i=1}^s (\by^*_i)^2\right\}} = \sqrt{\sum_{i=1}^s \Ex\left\{(\by^*_i)^2\right\}} \overset{\eqref{eq:2nd-moment}}{=} O\left(\kappa \sqrt{s \log(en/s)} \right),
$$
as required, where the above inequality holds by Jensen's inequality for concave functions, and to support the last equality we use Stirling approximation as to write
\begin{align*}
\sum_{i=1}^s \log(en/i) & =s\log(en) - \sum_{i=1}^s \log(i) = s\log(en) - \log(s!) \leq s \log(en) - s \log s + s = s \log(en/s) +s\\ 
& = O(s \log(en/s)).
\end{align*}

\medskip
To see that~\eqref{eq:2nd-moment} holds, commence with noticing that for every $t > 0$, every $i \in [s]$, and every $2 \leq p \leq \log (en)$, it holds that
\begin{align}
\Pr\left\{(\by^*_i)^2 \geq t \right\} & \leq \sum_{\substack{S \subseteq [n]\\ |S| = i}} \prod_{j \in S}\Pr\left\{ |\by_j|^2 \geq t\right\} \nonumber \\
& = \sum_{\substack{S \subseteq [n]\\ |S| = i}} \prod_{j \in S} \Pr\left\{ |\by_j|^p \geq t^{p/2}\right\} \nonumber \\
& \leq \sum_{\substack{S \subseteq [n]\\ |S| = i}} \prod_{j \in S} \frac{\|\by_i\|_p^p}{(\sqrt{t})^p} \nonumber \\
& \leq \binom{n}{i} \left(\frac{\kappa \sqrt{p}}{\sqrt{t}} \right)^{ip}, \label{eq:tail}
\end{align}
where the penultimate inequality holds by Markov's inequality, and the last inequality holds by the premise of the lemma.

For $i \in [s]$ for which $p := \lfloor \log(en/i)\rfloor \geq 2$ holds (i.e. $i \leq n/e$), and for $t = e^4 u \kappa^2 \log(en/i)$ with $u \geq 1$, we may proceed to write
\begin{equation}\label{eq:specific-tail}
\Pr\left\{(\by^*_i)^2 \geq e^4 u \kappa^2 \log(en/i) \right\}  \overset{\eqref{eq:tail}}{\leq} \left(\frac{en}{i}\right)^i \left(\frac{\kappa \sqrt{\log (en/i)}}{e^2 \sqrt{u} \kappa \sqrt{\log (en/i)}}\right)^{i \log (en/i)} \leq \left(\frac{1}{u}\right)^{i \log(en/i)/2}.
\end{equation}
For $i \in [s]$ such that $\lfloor \log(en/i) \rfloor =1$ (i.e. $i > n/e$), utilise~\eqref{eq:tail} with $p=2$ to obtain 
\begin{equation}\label{eq:specific-tail-2}
\Pr\left\{(\by^*_i)^2 \geq e^4 u \kappa^2 \log(en/i) \right\} \leq 2^n \left(\frac{2\kappa^2}{e^4 u \kappa^2 \log(en/i)} \right)^{i} \leq \left(\frac{1}{u}\right)^{n/e}.
\end{equation}
Then, for $i \leq n/e$, we have 
\begin{align*}
\Ex\left\{(\by^*_i)^2\right\} & \overset{\phantom{\eqref{eq:specific-tail}}}{=} \int_0^\infty \Pr\left\{ (\by^*_i)^2 \geq t\right\}\; \mathrm{d}t\\
& \overset{\phantom{\eqref{eq:specific-tail}}}{=} \int_0^{e^4 \kappa^2 \log(en/i)} \Pr\left\{ (\by^*_i)^2 \geq t\right\}\; \mathrm{d}t + \int_{e^4 \kappa^2 \log(en/i)}^\infty \Pr\left\{ (\by^*_i)^2 \geq t\right\}\; \mathrm{d}t\\
& \overset{\phantom{\eqref{eq:specific-tail}}}{=} O\left( \kappa^2 \log(en/i)\right)\left(1 + \int_1^\infty \Pr\left\{ (\by^*_i)^2 \geq e^4 u \kappa^2 \log(en/i)\right\}\; \mathrm{d}u\right)\\
& \overset{\eqref{eq:specific-tail}}{\leq} O\left( \kappa^2 \log(en/i)\right)\left(1 + \int_1^\infty \left(\frac{1}{u}\right)^{i \log(en/i)/2}\; \mathrm{d}u \right) \\
& \overset{\phantom{\eqref{eq:specific-tail}}}{=} O\left(\kappa^2 \log(en/i)\right),
\end{align*}
where for the last equality we rely on $i \log (en/i)/2 \geq 2$ holding which in turn compels the convergence $\int_1^\infty u^{-2}\;\mathrm{d}u = 1$. The required inequality $i \log (en/i)/2 \geq 2$ can be rewritten to read $n \geq i\cdot \exp(-1 + 4/i)$ which holds for sufficiently large $n$ since $i \leq s \leq n$. 

Similarly, for $i > n/e$, write 
\begin{align*}
\Ex\left\{(\by^*_i)^2\right\} & \overset{\phantom{\eqref{eq:specific-tail-2}}}{=} \int_0^\infty \Pr\left\{ (\by^*_i)^2 \geq t\right\}\; \mathrm{d}t\\
& \overset{\phantom{\eqref{eq:specific-tail-2}}}{=} \int_0^{e^4 \kappa^2 \log(en/i)} \Pr\left\{ (\by^*_i)^2 \geq t\right\}\; \mathrm{d}t + \int_{e^4 \kappa^2 \log(en/i)}^\infty \Pr\left\{ (\by^*_i)^2 \geq t\right\}\; \mathrm{d}t\\
& \overset{\phantom{\eqref{eq:specific-tail-2}}}{=} O\left( \kappa^2 \log(en/i)\right)\left(1 + \int_1^\infty \Pr\left\{ (\by^*_i)^2 \geq e^4 u \kappa^2 \log(en/i)\right\}\; \mathrm{d}u\right)\\
& \overset{\eqref{eq:specific-tail-2}}{\leq} O\left( \kappa^2 \log(en/i)\right)\left(1 + \int_1^\infty \left(\frac{1}{u}\right)^{n/e}\; \mathrm{d}u \right) \\
& \overset{\phantom{\eqref{eq:specific-tail-2}}}{=} O\left(\kappa^2 \log(en/i)\right).
\end{align*}
\hfill\QED
\end{document}